\documentclass[runningheads]{llncs}

\usepackage{amssymb,latexsym}
\usepackage{amsmath}
\usepackage{graphicx}
\usepackage{hyperref}
\usepackage{titletoc}
\usepackage{cite}
\numberwithin{equation}{section}

\newtheorem{prop}[theorem]{Proposition}

\renewcommand{\P}{\mathbb P}
\newcommand{\Z}{\mathbb Z}

\newcommand{\EE}{\mathcal E}

\newcommand{\N}{\mathbb N}

\newcommand{\dbf}{\mathbf{d}}
\newcommand{\T}{\mathbb{T}}
\newcommand{\Torus}{\mathbb{T}}
\newcommand{\Tl}{\mathfrak{T}}
\newcommand{\Tgw}{\mathrm{T}}
\newcommand{\tbf}{\mathbf{t}}
\newcommand{\Tbf}{\mathbf{T}}
\newcommand{\sbr}{\mathbf{s}}
\newcommand{\cI}{\mathcal{I}}
\newcommand{\cV}{\mathcal{V}}
\newcommand{\Pl}{\Xi}
\newcommand{\dTV}{\mathbf{d}_{\mathrm{TV}}}
\newcommand{\dout}{\mathbf{d}_\mathrm{o}}

\newcommand{\ssubset}{\subset\subset}
\newcommand{\Capa}{\mathrm{Cap}}
\newcommand{\BCap}{\mathrm{BCap}}
\newcommand{\Es}{\mathrm{Es}}
\newcommand{\Esc}{\mathrm{Esc}}

\newcommand{\Snake}{\mathcal S}

\newcommand{\dist}{\rho}

\newcommand{\musb}{\tilde{\mu}}
\newcommand{\mumo}{\mu_{-1}}
\newcommand{\degree}{\mathbf{d}}
\newcommand{\tc}{\Pi_{\mathrm{c}}}
\newcommand{\tv}{\Pi_{\mathrm{v}}}

\newcommand{\eqqc}{\cong_{\mathrm{c}}}

\newcommand{\shiftc}{\tau_{\mathrm{c}}}
\newcommand{\shiftv}{\tau_{\mathrm{v}}}

\newcommand{\shiftsc}{\tau^*_{\mathrm{c}}}
\newcommand{\shiftsv}{\tau^*_{\mathrm{v}}}
\newcommand{\stc}{\Pi^*_{\mathrm{c}}}
\newcommand{\stv}{\Pi^*_{\mathrm{v}}}
\newcommand{\nuc}{\nu_{\mathrm{c}}}
\newcommand{\nuv}{\nu_{\mathrm{v}}}
\newcommand{\Tfini}{\mathbb{T}_\mathrm{f}}
\newcommand{\Tinf}{\mathbb{T}_\infty}
\newcommand{\Tinfv}{\mathbb{T}'_{\infty}}
\newcommand{\Tinfs}{\mathbb{T}^*_\infty}
\newcommand{\Tinfsv}{\mathbb{T}^{\star}_\infty}
\newcommand{\Shift}{\tau}
\newcommand{\cc}{s}
\newcommand{\Bbs}{\mathbf{B}}
\newcommand{\Comp}{\mathrm{Comp}}

\begin{document}
\title{Branching interlacements and tree-indexed random walks in tori\thanks{The author is supported by ISF grant 1207/15 and ERC starting grant 676970}}
%
%
\author{Qingsan Zhu\inst{1}}
\authorrunning{Qingsan Zhu}
%
\institute{School of Mathematical Sciences, Tel Aviv University, Tel Aviv, 69978 ISRAEL
\email{qingsanz@mail.tau.ac.il}}
\maketitle              
\begin{abstract}
We introduce a model of branching interlacements for general critical offspring distributions. It consists of a countable collection of infinite tree-indexed random walk trajectories on $\Z^d$, $d\geq 5$. We show that this model turns out to be the local limit of the tree-indexed random walk in a discrete torus, conditioned on the size proportional to the volume of the torus. This generalizes the previous results of Angel, R\'{a}th and the author, for the critical geometric offspring distribution. Our model also includes the model of random interlacements introduced by Sznitman as a degenerate case.

To obtain the local convergence, we establish results on decomposing large random trees into small trees, local limits of random trees around a prefixed vertex, and asymptotics of the visiting probability of a set by a tree-indexed random walk with a given size in a torus. These auxiliary results are interesting in themselves. As another application, we show that when $d\geq 5$ the cover time of a $d$-dimensional torus of side-length $N$ by tree-indexed random walks is concentrated at $N^d\log N^d/\BCap(\{0\})$.

\end{abstract}
\section{Introduction}
This article introduces a model of branching interlacements which consists of a countable collection of infinite tree-indexed random walk trajectories on $\Z^d,d\geq5$. A non-negative parameter $u$ measures the amount of trajectories that enter the picture. The union of the range of these tree-indexed random walk trajectories defines the branching interlacement at level $u$. It is an infinite translation invariant random subset of $\Z^d$. This model is of special interest since it offers a microscopic description of the structure left by tree-indexed random walks. In fact, we show that the branching interlacement at level $u$ does appear as the local limiting distribution of the trace of the tree-indexed random walk in a large discrete torus, with a size proportional to the volume of the torus.

Let us first recall a similar model, the random interlacements model, which was introduced by Sznitman in \cite{S10V}. This model describes the local picture left by the trace of a random walk in a discrete torus when it runs up to times proportional to the volume of the torus. Roughly speaking, the model of random interlacements can be constructed via a Poisson point process with intensity measure $u\nu_0$. The measure $\nu_0$ is supported on the space of doubly infinite nearest-neighbour trajectories modulo time shifts on $\Z^d,d\geq3$. The parameter $u\geq 0$, referred to as the level, measures the amount of trajectories that enter the picture. The union of the range of trajectories contained in the support of this Poisson point process defines the random subset $I^u$ of $\Z^d$, called the random interlacement at level $u$. Its law can be characterized as the unique distribution on $\{0,1\}^{\Z^d}$ such that
\begin{equation}\label{RI-Cap}
P[I^u\cap K=\emptyset]=\exp(-u\cdot\Capa(K)), \mathrm{~for~every~finite~}K\subseteq \Z^d,
\end{equation}
where $\Capa(K)$ is the (discrete) capacity of $K$. In addition, there is an equivalent way to construct a set with the same law as $I^u\cap K$.
\begin{property}
   Let $N_K$ be a Poisson random variable with parameter $u\cdot\Capa(K)$, and $(X^j)_{j\geq 1}$ be i.i.d.\ random walks with the harmonic measure from infinity of $K$ as the initial measure. Then $K\cap\left(\cup_{j=1}^{N_K}\mathrm{Range}(X^j)\right)$ has the same distribution as $I^u\cap K$.
\end{property}

On the other hand, branching interlacements (only) for the critical geometric offspring distribution were constructed in \cite{ARZ15} (or see \cite{ZThesis}). The construction is essentially based on the fact that the so-called contour walk of the Galton-Watson tree with the critical offspring distribution is just the simple random walk which is Markovian. This article extends this construction to general critical offspring distributions, by introducing two shift transformations and the corresponding invariant measures on a certain set of infinite trees.

Before entering the formal construction of our model, let us emphasize that either of the two characterizations above for random interlacements can be used to characterize the branching interlacement at level $u$, except that we need to use the corresponding subjects for branching random walks. Branching capacity $\BCap(K)$ and the corresponding `harmonic (entering) measure from infinity' $m_K$ for any finite subset $K\ssubset \Z^d$ were introduced by the author in \cite{Z161} (or see \cite{ZThesis}). They are analogues to capacity and harmonic measure of random walk, in the setting of critical branching random walks. We write $K\ssubset \Z^d$ to indicate that $K$ is a finite subset of $\Z^d$.
\begin{theorem}
For any $u>0$, we can define a random subset of $\Z^d$,  called the branching interlacement at level $u$ and denoted by $\cI^u$. Its law is characterized by
\begin{equation}\label{BI-Bcap}
P[\cI^u\cap K=\emptyset]=\exp(-u\cdot\BCap(K)), \mathrm{~for~any~}K\ssubset \Z^d.
\end{equation}
\end{theorem}
Similarly, we have the following.
\begin{prop}\label{BI-mK}
   Let $N_K$ be a Poisson random variable with parameter $u\cdot\BCap(K)$, and $(X^j)_{j\geq 1}$ be i.i.d.\ branching random walks with $m_K$ as the initial measure. Then $K\cap\left(\cup_{j=1}^{N_K}\mathrm{Range}(X^j)\right)$ has the same distribution as $\cI^u\cap K$.
\end{prop}

We now describe our model. The model of branching interlacements is constructed via a Poisson point process on the space of infinite tree-indexed random walks modulo an equivalence relation. Hence, the key step for the construction is to produce a measure on that space as the intensity measure.

We consider tree-indexed random walks, also called random spatial trees or branching random walks. Fix a probability measure $\mu$ on $\N$ which serves as the offspring distribution (for the random tree mechanism) and a probability measure $\theta$ on $\Z^d$ which serves as the jump distribution (for the random walk mechanism). We assume that $\mu$ is nondegenerate (i.e.\ not the Dirac mass at $1$) with mean one and finite variance, and that $\theta$ is with mean zero and finite range, not supported on any strict subgroup of $\Z^d$. First, we introduce a probability measure $\tc$ on the space of infinite rooted ordered trees. In fact, this measure is supported on the subset, denoted by $\Tinf$, of those trees which have a unique infinite ray from the root. Such infinite rays are called spines. The measure $\tc$ can be uniquely determined by the following property. For any $\tbf\in\Tinf$ and a non-root vertex $v_0$ of $\tbf$ in the spine, the probability under $\tc$ of being equal to $\tbf$ up to vertex $v_0$ is
\begin{equation}\label{tc-k}
\frac{1}{2}\prod_{v}\mu(\degree(v)-1),
\end{equation}
The product is over all vertices that are in the same component of the root, when deleting $v_0$; $\degree(v)$ is the degree of $v$; $1/2$ is used for a normalization, i.e.\ to make $\tc$ a probability measure. We still need to define an equivalence relation, denoted by $\eqqc$. Informally, $\tbf_1\eqqc \tbf_2$ if and only if $\tbf_2$ is $\tbf_1$ re-rooted at another corner. Figure 1 explains how to re-root a tree at different corners. A key feature of $\tc$ is that this measure is invariant under the corner shift, which can be easily verified using \eqref{tc-k}.
\begin{figure*}[h]
	\includegraphics[width=4.9in]{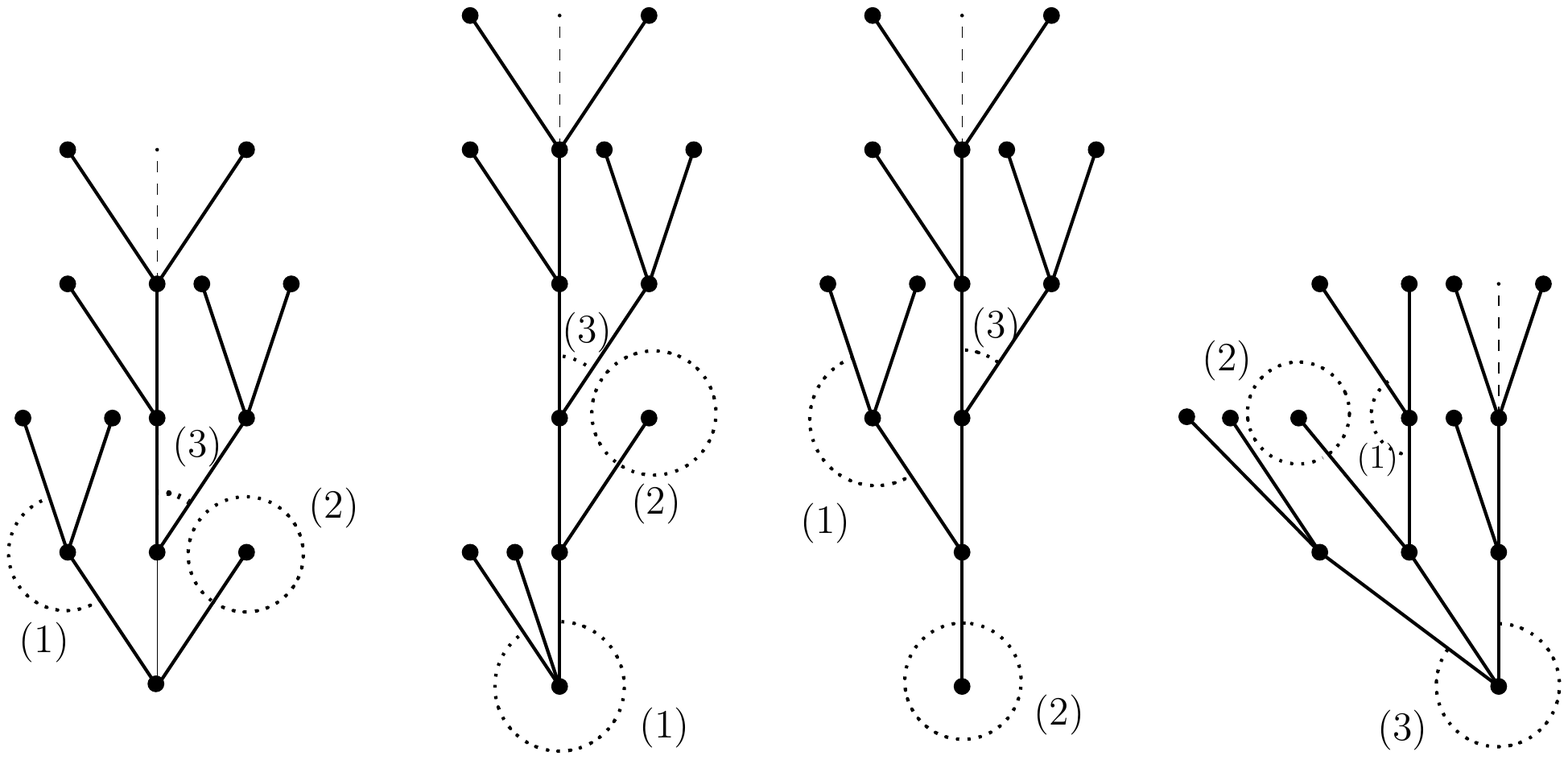}
	\caption{Re-root a tree at different corners}
\end{figure*}

We now introduce a measure on the set of spatial trees. Let $\Tinfs$ be the set of all pairs $(\tbf,\Snake_\tbf)$, where $\tbf\in\Tinf$ and $\Snake_\tbf:\tbf\rightarrow \Z^d$ is a map from the vertex set $\tbf$ into $\Z^d$. For simplicity, we also write $\tbf$ for the vertex set of tree $\tbf$. Similarly, we can define a measure $\stc$ as follows. For any $(\tbf,\Snake_\tbf)\in\Tinfs$ and one non-root vertex $v_0$ of $\tbf$ in the spine, the total measure, under $\stc$, of the set of all those spatial trees that are equal to $(\tbf,\Snake_\tbf)$ up to $v_0$ is
\begin{equation}\label{stc-k}
\frac{1}{2}\prod_{v}\mu(\degree(v)-1)\prod_{(v,v')}\theta(\Snake_\tbf(v')-\Snake_\tbf(v)).
\end{equation}
The first product is, as before, over all vertices that are in the same component as the root, when deleting $v_0$; the second product is over all edges at this component, oriented in such a way that $v$ is closer to $v_0$ than $v'$. Note that $\stc$ is not a probability measure, but a $\sigma$-finite measure. The mass of the set of all spatial trees that send the root to a prefixed point is one. Similarly, we can define an equivalence relation between spatial trees, also denoted by $\eqqc$. $(\tbf_1,\Snake_{\tbf_1})\eqqc(\tbf_2,\Snake_{\tbf_2})$ if and only if $\tbf_1\eqqc \tbf_2$ and as maps, $\Snake_{\tbf_1}=\Snake_{\tbf_2}$. Note that when $\tbf_1\eqqc \tbf_2$, $\tbf_2$ has the same vertex set as $\tbf_1$ via a graph isomorphism. Therefore, $\Snake_{\tbf_1}$ and $\Snake_{\tbf_2}$ will have the same domain under this isomorphism. Similarly to the tree case, since both products are independent of the choice of the corner rooted, one can see that $\stc$ is invariant under the corner shift.

Due to this invariance property, we can obtain a measure $\nuc$ on $\overline{W}=\Tinfs/\eqqc$, induced by $\stc$. The branching interlacement at level $u$ is governed by a Poisson point process $\omega=\sum_{i\geq 0}\delta_{\overline{w}_i}$ on $\overline{W}$, with intensity measure $2u\nuc$. Denote by $\P^u$ the law which turns $\omega$ into a Poisson point process with intensity $2u\nuc$. Then the branching interlacement at level $u$ is defined by
$$
\cI^u(\omega)=\bigcup_{i\geq 0}\mathrm{Range}(\overline{w}_i),\;\mathrm{if}\,\omega=\sum_{i\geq 0}\delta_{\overline{w}_i},
$$
where $\mathrm{Range}(\overline{w}_i)$ is the range of any element in the equivalence class of $\overline{w}_i$. The vacant set at level $u$ is $\cV^u=\Z^d\setminus\cI^u$.

\begin{remark}
The reason why we put the constant factor ``$2$'' in the intensity $2u\nuc$ is to make \eqref{BI-Bcap} consistent with \eqref{RI-Cap}. In fact, for any  $A\ssubset\Z^d$, $\nuc(\overline{W}_A)=\BCap(A)/2$, where $\overline{W}_A$ is the set of all equivalence classes of those spatial trees that intersect $A$, cf.\ \eqref{def-WA}. Note that $\P^u[\cI^u \cap A=\emptyset]=\exp(-2u\nuc(\overline{W}_A))$ then \eqref{BI-Bcap} follows. Moreover, we will define another invariant measure $\stv$ on (a proper subset of) $\Tinfs$, and thus obtain $\nuv$ on the space of equivalence classes of spatial trees. It turns out that $\nuv(\overline{W}_A)=2\nuc(\overline{W}_A)$ for any $A\ssubset \Z^d$. Hence, using $\nuv$ instead of $2\nuc$, we also obtain the branching interlacement set with the same distribution, cf.\ Section 2. The factor ``$2$'' comes from the fact that on average, a vertex in a large tree has two corners.
\end{remark}

Until now, we have not used the assumption that $d\geq 5$. In fact, we need this assumption to guarantee that when $u>0$ $\cI^u$ is nontrivial, i.e.\ a nonempty proper subset of $\Z^d$. It is showed in \cite{Z162} that the spatial tree with law $\stc$ (at least when $\mu$ has finite third moment and $\theta$ is symmetric) almost surely visits every vertex on $\Z^d$ for $d\leq 4$. On the other hand, on $\Z^d,d\geq 5$, every finite subset is visited finitely many times almost surely. It is worth pointing out that most of our construction can be applied to a transient tree-indexed random walk attached to an infinite locally finite connected graph, see Remark~\ref{rm1}.10.

Similarly to random interlacements, the branching interlacement set $\cI^u$ presents a long range dependence:
$$
\mathrm{Cov}(1_{x\in \cI^u},1_{y\in \cI^u})\sim \frac{cu}{\|x-y\|^{d-4}}, \mathrm{~as~} \|x-y\|\rightarrow \infty.
$$

We now turn to our main result that the branching interlacement gives the local limit of the tree-indexed random walk in a large torus, conditioned on the size proportional to the volume of the torus. By a tree-indexed random walk with size $n$ and starting point $x\in\Z^d$, we mean a random spatial tree $(\tbf,\Snake_\tbf)$ as follows. First, $\tbf$ is a Galton-Watson tree with offspring distribution $\mu$, conditioned on the total number of vertices being $n$. Second, conditionally on $\tbf$, the conditional probability weight of $\Snake_\tbf$ is determined by
$$
1_{\Snake_\tbf(o)=x}\cdot\prod_{(v,v')}\theta(\Snake_\tbf(v')-\Snake_\tbf(v)),
$$
where $o$ is the root and the product is over all edges, oriented in such a way that $v$ is closer to the root than $v'$. In fact, we will assume that $\theta$ is symmetric for the local limit result, and then the orientation is not needed. Denote by $\varphi=\varphi_N:\Z^d\rightarrow \Torus_N=(\Z/N\Z)^d$ the canonical projection map induced by $\mathrm{mod}~N$. Via this $\varphi$, one can define the tree-indexed random walk in $\Torus_N$ with size $n$ and a uniform starting point in an obvious way. In Theorem~\ref{ll-thm-main}, we show that assuming further that $\mu$ has finite exponential moments and that $\theta$ is symmetric and aperiodic, for any $K\ssubset \Z^d$, $u\in(0,\infty)$ and $n=n(N)$, an integer-valued function satisfying $\lim_{N\rightarrow \infty}n(N)/N^d=u$,
$$
\lim_{N\rightarrow \infty}P[\mathrm{Range}(\Snake^{n,N})\cap \varphi(K)=\emptyset]=\exp(-u\BCap(K)),
$$
where $\Snake^{n,N}$ is a tree-indexed random walk in $\Torus_N$ with size $n$ and a uniform starting point. By the inclusion-exclusion principle, this result implies the local convergence of the trace left by a tree-indexed random walk.

\begin{theorem}\label{thm-lc}
Under the same assumptions above, for any $A\subset K\ssubset\Z^d$, we have
$$
\lim_{N\rightarrow \infty}P[\mathrm{Range}(\Snake^{n,N})\cap \varphi(K)=\varphi(A)]=P[\cI^u\cap K=A].
$$
\end{theorem}

We now briefly describe our idea to show the local convergence. The method is similar to the one for ``the law of rare events". Similarly to the case of the geometric distribution in \cite{ARZ15}, two main ingredients are needed. One is to find exact asymptotics for the visiting probability of a set in a torus by a tree-indexed random walk with a small size (cf.\ Theorem~\ref{visit-small}); the other is to decompose a large random tree into small random trees (cf.\ Theorem~\ref{cutting}). As mentioned earlier, in the geometric case, the contour walk of the random tree is just the simple random walk. Note that a (simple) random walk is Markovian, and that exact counting of paths is usually achievable for simple random walk. Moreover, in the geometric case, the corresponding random tree with a given size is invariant under the corner shift. In our case, we do not have such nice properties. This creates serious difficulties when one tries to establish the corresponding results. To overcome such difficulties, we switch to the so-called Lukasiewisz walks which still have the Markov property. One might compare the proofs in this article with the ones in \cite{ARZ15}.

It is worth pointing out that we also establish two intermediate results on random trees, which are interesting in their own right. On the one hand, as we mentioned earlier, we show (cf.\ Theorem~\ref{cutting}) that, with high probability, a large random tree, precisely, a Galton-Watson tree conditioned on the total size, can be decomposed into small random trees with the same distributions as the Galton-Watson trees conditioned on size, without losing too many vertices. We can even require that those small subtrees are relatively far from each other.

On the other hand, since we aim to establish the local limit of tree-indexed random walks, we first establish a result (cf.\ Theorem~\ref{thm-ll}) for the local limit of large random trees, i.e.\ the local picture around a vertex which is not too close to the root. It is well-known that the local limit around the root is given by the so-called Galton-Watson tree conditioned on survival. Local limits of large random trees have been studied in \cite{A91A,S18L}. However, the local limits constructed there are around a uniformly selected vertex in the random tree, while we consider the local limit around any prefixed vertex as long as it is not too close to the root, say, the $m$-th vertex in the Depth-first search order when $m$ is not too small or too big. We also give bounds for the error terms.

Finally, as another application of our results on random trees and tree-indexed random walks in tori, we obtain a result (Theorem~\ref{thm-cover}) on the cover time of a torus by tree-indexed random walks. Cover times of finite graphs by a simple random walk have been studied extensively (see e.g.\ \cite{A83O, A91T, BH91C, DPRZ04C, B13G}). One important example is the cover time of the discrete torus. The cover time $\mathrm{Cov}_N$ of the torus $\Torus_N=(\Z/N\Z^d)^d$ is the time taken for a simple random walk to visit every vertex of $\Torus_N$. It is well-known that when $d\geq 3$, $E[\mathrm{Cov}_N]\sim g(0)N^d\log N^d$ where $g$ is the discrete Green function, and $\mathrm{Cov}_N/(g(0)N^d\log N^d)\rightarrow 1$ in probability. Recently, by constructing a strong coupling between random walks in tori and random interlacements, it is proved that the fluctuations of $\mathrm{Cov}_N$ are governed by the so-called Gumbel distribution \cite{B13G}. The above is for cover times by random walks. However, we do not find any reference for cover times by tree-indexed random walks. It is even vague to define the cover time by tree-indexed random walks, since unlike random walk, there does not exist a simple way to couple tree-indexed random walks (random trees) with different sizes such that the smaller one is embedded into the bigger one. It is even not well-understood whether there exists such a coupling or not (see e.g.\ Problem 1.15 in \cite{J06R} and \cite{LW04B,J06C}). Nevertheless, we show that for $d\geq 5$, the cover time of $\Torus_N$ is concentrated at $N^d\log N^d/\BCap(\{0\})$ in the following sense. For any $\epsilon>0$, the probability that every vertex in $\Torus_N$ is visited by a tree-indexed random walk with size $n=n(N)$ goes to one if $n(N)\geq (1+\epsilon)N^d\log N^d/\BCap(\{0\})$, and zero if $n(N)\leq (1-\epsilon)N^d\log N^d/\BCap(\{0\})$, as $N\rightarrow \infty$. Note that $g(0)N^d\log N^d=N^d\log N^d/\Capa(\{0\})$, and notice the obvious analogy between the results for a random walk and a tree-index random walk.

There are many natural questions about branching interlacements that are untouched in this article. For example, does the vacant set percolate or not? The answer is similar to the random interlacement case: there exists a critical value $u^*\in(0,\infty)$, such that when $u<u^*$ it does and when $u>u^*$, it does not, see the forthcoming paper \cite{Z191}. Moreover, it is probable that by constructing suitable couplings between tree-indexed random walks in tori and branching interlacements, we could obtain new results on cover times as well as the largest component in the vacant set of the torus. Since the model of random interlacements has been constructed, a great effort has been made and many beautiful results have been established in that model itself and its relation with random walk (see e.g.\ \cite{SS09P, S12D, PT15S, B13G, RS11O, SS10C, TW11O} and the references therein). It seems that many results for random interlacements can be generalized to our model.

The article is organized as follows. In Section 2, we construct the model of branching interlacements. The main task is to construct the measures $\nuc$ and $\nuv$ entering the intensity of the Poisson point process that we are after. In Section 3, we consider random trees and establish our results of the local limit and of the decomposition of large random trees. Section 4 deals with asymptotics of the visiting probability by a tree-indexed random walk in a torus and the local convergence of tree-indexed random walks. Section 5 is devoted to the cover time result.

We finish this section with a remark on constants and notations. Throughout the text, we use $C,c,C_1,c_2$ etc.\ to denote positive constants depending only on the dimension $d$, the offspring distribution $\mu$ and the jump distribution $\theta$. These constants may change from place to place. Dependence of constants on additional parameters will be made or stated explicitly. For example, $C(\lambda)$ stands for a positive constant depending on $d, \mu,\theta$ and $\lambda$. For functions $f(x)$ and $g(x)$, we write $f\sim g$ if $\lim_{x\rightarrow \infty}(f(x)/g(x))=1$. We write $f\preceq g$ and $f\succeq g$, respectively, if there exists a constant $C$ such that, $f\leq Cg$ and $f\geq Cg$ . We use $f\asymp g$ to express that $f\preceq g$ and $f\succeq g$. We write $f\ll g$ if $\lim_{x\rightarrow \infty}(f(x)/g(x))=0$. We use the notation $[[a,b]]=[a,b]\cap \Z$ for $a\leq b$. For any finite or infinite set $A$, we write either $|A|$ or $\sharp A$ for the cardinality of $A$. We write $a\wedge b$ for $\min\{a,b\}$.

\section{Branching interlacements}

\subsection{Finite and infinite trees}
We consider rooted ordered trees (also called plane trees). Recall that a tree is rooted if one vertex is distinguished as the root $o$. If $v\neq o$ is a non-root vertex, then the parent of $v$ is just the unique neighbour of $v$ closer to $o$ than $v$ (hence $o$ has no parent). Conversely, for any vertex $v$, the neighbours of $v$ that are further away from $o$ than $v$ are the children of $v$. The number of children of $v$ is the out-degree $\dout(v)$ of $v$. We write $\degree(v)$ for the regular degree of $v$. Recall further that a rooted tree is ordered if the children of each vertex are ordered.

We are interested in Galton-Watson trees (GW-trees) and their companions. Given a probability measure $\mu$ on $\N$, the GW-tree (with offspring distribution $\mu$) can be defined recursively, starting with the root and then giving each vertex, independently, a random number of children according to $\mu$. Throughout this work, we fix $\mu$ and always assume that $\mu$ is critical (i.e.\ $\sum_{i\in\N }i\mu(i)=1$) and with finite variance $\sigma^2>0$.

We also need to consider the so-called GW-tree conditioned on survival, denoted by $\Tbf^\infty$, defined as follows. \begin{itemize}
    \item Each vertex is normal or special.
    \item The root is special.
    \item A normal vertex produces only normal individuals, independently, according to $\mu$.
    \item A special vertex produces individuals independently, according to the so called size-biased distribution $\musb$ (i.e. $\musb(k)=k\mu(k)$). One of them, chosen uniformly at random, is special; the others (if any) are normal.
\end{itemize}
For convenience, we denote the corresponding sample space by
\begin{multline*}
    \Tinf=\{\tbf:\tbf\mathrm{~has~a~unique~infinite~path~starting~from~the~root;}\\
    \mathrm{there~are~infinite~vertices~on~each~side~of~the~infinite~path}\}.
\end{multline*}
For $\tbf\in\Tinf$, the unique infinite path starting from the root is called the spine of $\tbf$. Note that for any $\tbf\in\Tinf$, the depth-first search (from the root) explores only those vertices on the left hand side of $\tbf$, i.e.\ the vertices that are either in the spine or the descendants of those vertices which are elder siblings of the spine vertices (any vertex is regarded as a descendant and an ancestor of itself). For any $\tbf\in\Tinf$, we do have a full order, called the depth-first search order from infinity, on the vertex set $\tbf$ as follows. First divide $\tbf$ into two subsets $A_1$ and $A_2$. Let $A_1$ be the set of those vertices that are strictly on the left hand side of $\tbf$, i.e.\ the descendants of those vertices which are elder siblings of the spine vertices; let $A_2=\tbf\setminus A_1$. Any vertex in $A_1$ goes after any vertex in $A_2$. The order inside $A_1$ is consistent with the depth-first search order from the root. For any $v_1\neq v_2\in A_2$, let $v_1'$ ($v_2'$) be the furthest vertex in the spine, from the root,  that is an ancestor of $v_1$ ($v_2$). Then $v_1$ is before $v_2$ if and only if either $v_1'$ is a strict descendant of $v_2'$, or $v_1'=v_2'$ and $v_1$ goes before $v_2$ in the depth-first search order from $v_1'$ in the subtree grafted to the right hand side of $v_1'$, i.e.\ the subtree generated by those vertices on the right hand side of the spine that are descendants of $v_1'$ but not descendants of the next spine vertex after $v_1'$. Intuitively, the depth-first search order from infinity is just the order of traversals by a particle from infinity travelling the tree along the edges from the right to the left. See Figure 3 for an illustration of this order.

\subsection{Invariant measures on infinite trees}
Let $\Tbf^\mathrm{c}$ be the random tree defined as follows.
\begin{itemize}
    \item Each vertex, except the root, is normal or special.
    \item The root produces $i$ individuals with probability $i\mu(i-1)/2$. One of them, chosen uniformly at random, is special; the others (if any) are normal.
    \item A normal vertex produces only normal individuals, independently, according to $\mu$.
    \item A special vertex produces individuals independently, according to the size-biased distribution $\musb$. One of them, chosen uniformly at random, is special; the others (if any) are normal.
\end{itemize}

Denote by $\tc$ the law of $\Tbf^\mathrm{c}$. For any $\tbf\in\Tinf$ and any non-root vertex $v$ of $\tbf$ in the spine, $v$ divides $\tbf$ into a finite number of components (subtrees). Write $\Comp(\tbf, v)$ for the component containing the root. For any $\tbf_1,\tbf_2\in \Tinf$ and one non-root vertex $v_2$ in the spine of $\tbf_2$, we say $\tbf_1$ is equal to $\tbf_2$ up to $v_2$, if there exists a vertex $v_1$ in the spine of $\tbf_1$, such that there exists a plane tree isomorphism (keeping the graph, the root and the order structures) between $\Comp(\tbf_1,v_1)$ and $\Comp(\tbf_2,v_2)$, sending $v_1$ to $v_2$. One can easily check that for any $\tbf\in\Tinf$ and $v_0$, a non-root vertex in the spine of $\tbf$,
\begin{equation}\label{eq-tc-k}
P[\Tbf^\mathrm{c} \mathrm{~is~equal~to~} \tbf \mathrm{~up~to~} v_0]=\frac{1}{2}\prod_{v}\mu(\degree(v)-1),
\end{equation}
where the product is over all vertices in $\Comp(\tbf,v_0)$ excluding $v_0$.

We now define the corner shift $\shiftc$ on $\Tinf$. Intuitively, $\shiftc(\tbf)$ is $\tbf$ re-rooted at the next corner. The definition of corners is obvious when we draw a plane tree `standardly' in a plane (see Figure 1 or Figure 2). Note that a vertex with degree $k$ has $k$ corners (around it). Formally, for $\tbf\in\Tinf$, its image $\shiftc(\tbf)$ can be obtained as follows. Let $v$ be the first child of the root $o$ in $\tbf$. Re-root $\tbf$ at $v$ and regard $o$ as the last child of $v$. Then this new tree (with root $v$) is just $\shiftc(\tbf)$. Note that $\shiftc$ is a bijection on $\Tinf$ and that $\shiftc^{-1}(\tbf)$ is $\tbf$ re-rooted at the previous corner. Moreover, all corners of $\tbf$ can be ordered in a two-sided sequence $(\mathbf{c}_n)_{n\in\Z}$. Informally, we image a particle that starts from the root, and then explores the left hand side of the tree via the edges. The sequence of traversals gives $(\mathbf{c}_n)_{n\geq0}$. Imaging a particle exploring the right hand side of the tree gives $(\mathbf{c}_{-n})_{n\geq0}$. Then for any corner $\mathbf{c}=\mathbf{c}_i$, we define $\shiftc(\tbf,\mathbf{c})$, called $\tbf$ re-rooted at $\mathbf{c}$, to be $\shiftc^{(i)}(\tbf)$. One could draw $\shiftc(\tbf,\mathbf{c})$ easily from the drawing of $\tbf$ and the location of $\mathbf{c}$ (see Figure 1 and Figure 2 for illustrations).

\begin{figure*}[h]
	\includegraphics[width=4.9in]{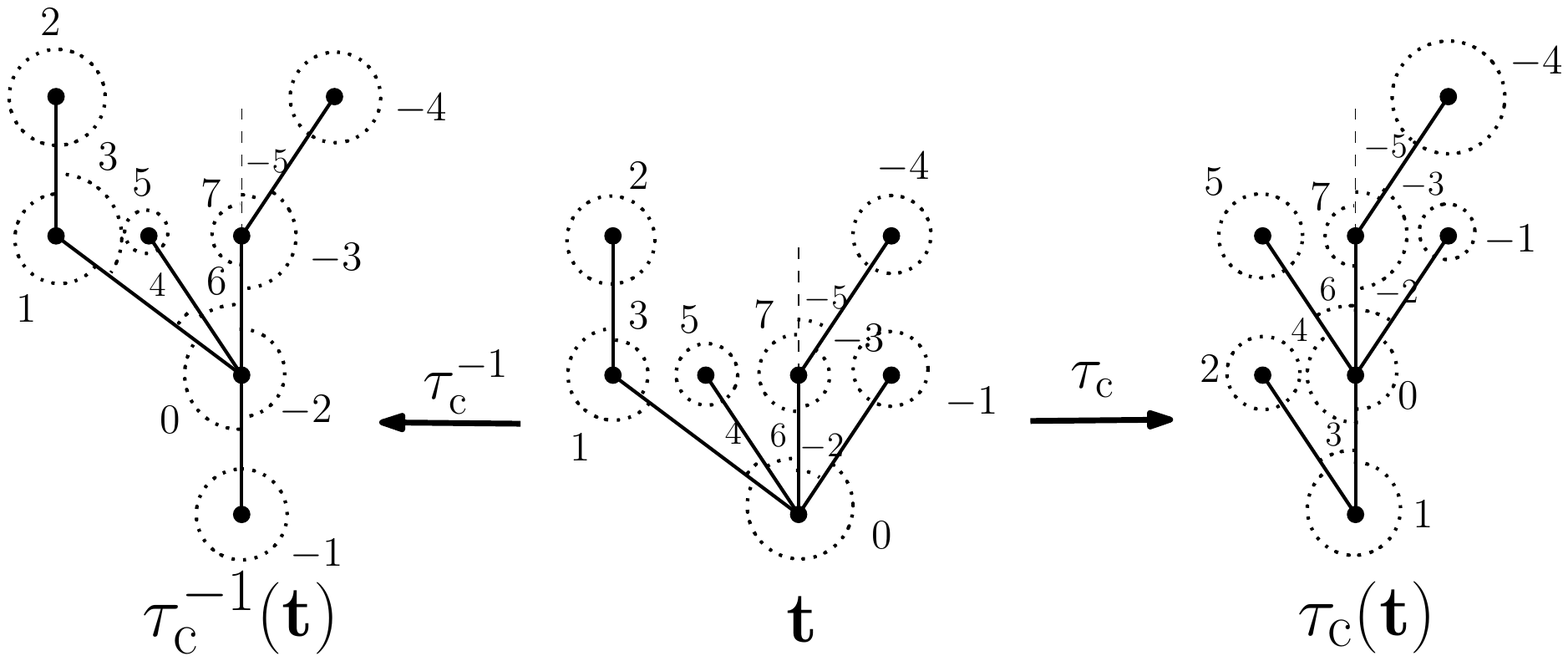}
	\caption{An illustration of the corner shift}
\end{figure*}

From \eqref{eq-tc-k}, we obtain
\begin{prop}
$\tc$ is invariant under $\shiftc$.
\end{prop}

We now turn to another probability measure $\tv$ and the corresponding shift $\shiftv$. Let $\Tbf^\mathrm{v}$ ($\tv$) be (the law of) the random tree defined as follows.
\begin{itemize}
    \item Each vertex, except the root, is normal or special.
    \item The root produces $i$ individuals with probability $\mu(i-1)$. The last individual is special; the others (if any) are normal.
    \item A normal vertex produces only normal individuals, independently, according to $\mu$.
    \item A special vertex produces individuals independently, according to the size-biased distribution $\musb$. One of them, chosen uniformly at random, is special; the others (if any) are normal.
\end{itemize}
Note that $\tv$ is supported on a proper subset of $\Tinf$, denoted by $\Tinfv$, consisting of those trees in which the last child of the root is the one in the spine. Note that for any $\tbf\in\Tinf$, there is a unique corner around $o$, such that the tree re-rooted at this corner belongs to $\Tinfv$. Similarly to \eqref{eq-tc-k}, for any $\tbf\in\Tinfv$ and $v_0$, a non-root vertex in the spine of $\tbf$,
\begin{equation}\label{eq-tv-k}
P[\Tbf^\mathrm{v} \mathrm{~is~equal~to~} \tbf \mathrm{~up~to~} v_0]=\prod_{v}\mu(\degree(v)-1),
\end{equation}
where the product is over all vertices in $\Comp(\tbf,v_0)$ excluding $v_0$.

For any $\tbf\in \Tinfv$ and a vertex $v\in \tbf$, it is elementary to see that there is a unique corner around $v$, such that the tree $\tbf$ re-rooted at this corner belongs to $\Tinfv$. We call this new tree $\tbf$ re-rooted at $v$. For any $\tbf\in\Tinfv$, we define $\shiftv(\tbf)$ to be the tree $\tbf$ re-rooted at the next vertex due to the Depth-first search order from infinity (this vertex is also the first vertex that is not in the spine, due to the Depth-first search order from the root). Note that $\tv$ is a bijection on $\Tinfv$ and $\tv^{-1}(\tbf)$ is just $\tbf$ re-rooted at the previous vertex due to the Depth-first search order from infinity. Therefore we can define $\shiftv^{(n)}$ for every $n\in\Z$. Similarly, since the right hand side of \eqref{eq-tv-k} is independent of the choice of the root, we have
\begin{prop}
$\tv$ is invariant under $\shiftv$.
\end{prop}

\begin{figure*}[h]
	\includegraphics[width=4.8in]{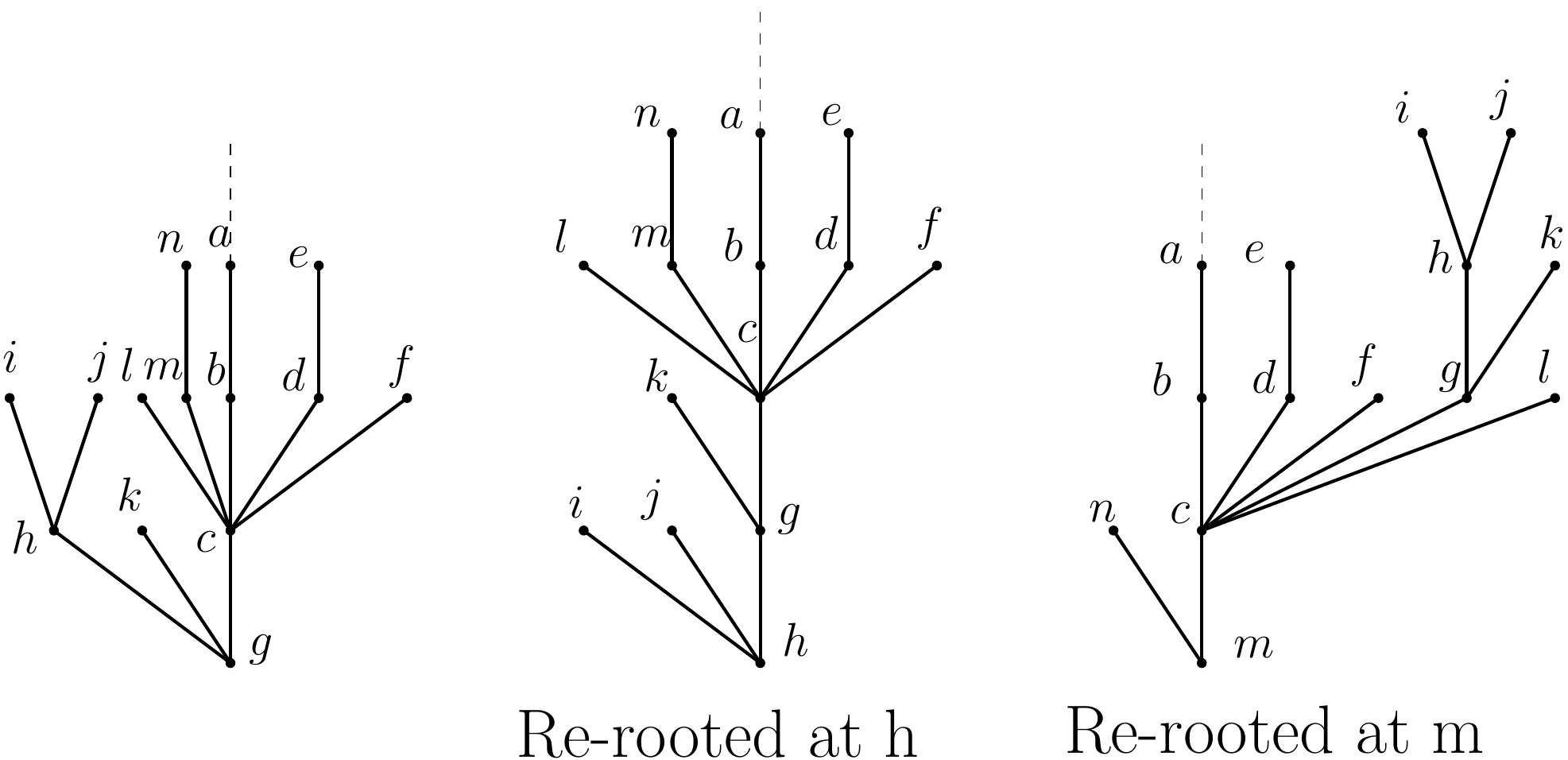}
	\caption{An illustration of the depth-first search order from infinity and the vertex shift.}
\end{figure*}

\begin{remark}
The vertex shift $\shiftv$ and the corresponding measure $\tv$ are related to the ones in \cite{LL16T} with the difference that only subtrees on the left side of the spine are considered there.
\end{remark}

\subsection{Random walk indexed by an infinite tree}
Let $\theta$ be a probability measure on $\Z^d,d\geq 5$. Throughout this work, we fix $\theta$ and always assume that $\theta$ is centered (i.e.\ with mean zero), with finite range and not supported in a strict subgroup of $\Z^d$. Let $\mathbf{t}\in \T_\infty$. Since we have not assumed that $\theta$ is symmetric, we first need to specify the orientation of edges. Intuitively, the edge is oriented in such a way that the starting point is always closer to infinity than the end point. Formally, for the edges in the spine, the orientation is from the child to the parent; for all other edges, the orientation is from the parent to the child. We write $\EE(\mathbf{t})$ for the set of all oriented edges of $\mathbf{t}$. The random walk indexed by $\mathbf{t}$, starting at $x\in \Z^d$ (with jump distribution $\theta$) is a random function $\Snake:\tbf\rightarrow \Z^d$, such that the root is mapped to $x$ and the random variables $(\Snake(v_2)-\Snake(v_1))_{(v_1,v_2)\in\EE(\mathbf{t})}$ are independent and distributed according to $\theta$. Write $P^x_\tbf$ for the law of the random walk indexed by $\tbf$ starting at $x$.

Let $\Tinfs$ ($\Tinfsv$) be the set of all pairs $(\tbf, \Snake)$ where $\tbf\in \Tinf$ ($\tbf\in \Tinfv$) and $\Snake:\tbf\rightarrow \Z^d$. We define $P^x_\mathrm{c}$ ($P^x_\mathrm{v}$) by declaring that $P^x_\mathrm{c}$ ($P^x_\mathrm{v}$) is the law of $(T,\Snake)$, where $T$ is distributed according to $\tc$($\tv$) and conditionally on $T=\tbf$, $\Snake$ is distributed according to $P^x_\tbf$.

For any $(\tbf_1,\Snake_1),(\tbf_2,\Snake_2)\in\Tinf$ ($\Tinfv$) and one non-root vertex $v_2$ of $\tbf_2$ in the spine, we say $(\tbf_1,\Snake_1)$ is equal to $(\tbf_2,\Snake_2)$ up to $v_2$, if there exist a vertex $v_1$ in the spine of $\tbf_1$  and a plane tree isomorphism $f$ between $\Comp(\tbf_2,v_2)$ and $\Comp(\tbf_1,v_1)$, such that $f(v_2)=v_1$ and $\Snake_2(v)=\Snake_1(f(v))$ for any $v$ in $\Comp(\tbf_2,v_2)$. From the definition of $P^x_\mathrm{c}$ ($P^x_\mathrm{v}$), one can see that, for any $(\tbf,\Snake_0)\in\Tinf$ or $\Tinfv$ and one non-root vertex $v_0$ of $\tbf$ in the spine,
\begin{eqnarray*}
P^x_\mathrm{c}[(T,\Snake) \mathrm{~is~equal~to~} (\tbf,\Snake_0) \mathrm{~up~to~} v_0]=\frac{1}{2}\prod_{v}\mu(\degree(v)-1)\prod_{(v,v')}\theta(\Snake_0(v')-\Snake_0(v)),\\
P^x_\mathrm{v}[(T,\Snake) \mathrm{~is~equal~to~} (\tbf,\Snake_0) \mathrm{~up~to~} v_0]=\prod_{v}\mu(\degree(v)-1)\prod_{(v,v')}\theta(\Snake_0(v')-\Snake_0(v)).
\end{eqnarray*}
In each formula, the first product is over all vertices in $\Comp(\tbf,v_0)$ excluding $v_0$, and the second is over all oriented edges in $\Comp(\tbf,v_0)$.

We now define the corresponding shift transformation $\shiftsc$ ($\shiftsv$) on $\Tinfs$ ($\Tinfsv$). For $\tbf^*=(\tbf, \Snake_0)\in \Tinfs (\Tinfsv)$, set $\shiftsc(\tbf^*)=(\tbf_1, \Snake_1)$ ($\shiftsv(\tbf^*)=(\tbf_2, \Snake_2)$), where $\tbf_1=\shiftc(\tbf)$ ($\tbf_2=\shiftv(\tbf)$) and keep the spatial locations of $\tbf_1$ ($\tbf_2$). In fact, $\shiftsc(\tbf^*)$ is $\tbf^*$ re-rooted at the next corner and $\shiftsv(\tbf^*)$ is $\tbf^*$ re-rooted at the next vertex (under the Depth-first search order from infinity). Note that the right hand sides of the formulas in the last paragraph are independent of the choice of the corner or the vertex rooted. Hence we have (when both the old root and the new root are in $\Comp(\tbf,v)$ excluding $v$)
\begin{align*}
P^{\Snake_0(o)}_\mathrm{c}[(T,\Snake) \mathrm{~is~equal~to~}& (\tbf,\Snake_0) \mathrm{~up~to~}v]\\
=&P^{\Snake_0(o_1)}_\mathrm{c}[(T,\Snake) \mathrm{~is~equal~to~} \shiftsc((\tbf,\Snake_0)) \mathrm{~up~to~} v],\\
P^{\Snake_0(o)}_\mathrm{v}[(T,\Snake) \mathrm{~is~equal~to~}& (\tbf,\Snake_0) \mathrm{~up~to~}v]\\
=&P^{\Snake_0(o_2)}_\mathrm{v}[(T,\Snake) \mathrm{~is~equal~to~} \shiftsv((\tbf,\Snake_0)) \mathrm{~up~to~} v],
\end{align*}
where $o,o_1,o_2$ are respectively, the roots of $\tbf,\shiftc(\tbf),\shiftv(\tbf)$. Hence we have
\begin{prop}\label{invariance}
$(P^x_\mathrm{c})_{x\in \Z^d}$ ($(P^x_\mathrm{v})_{x\in \Z^d}$) is invariant under $\shiftsc$ ($\shiftsv$). Precisely, for any measurable subset $A$ of $\Tinfs$  or $\Tinfsv$,
$$
\sum_{x\in\Z^d}P^x_\mathrm{c}(A)=\sum_{x\in\Z^d}P^x_\mathrm{c}(\shiftsc(A)), \text{ or }\sum_{x\in\Z^d}P^x_\mathrm{v}(A)=\sum_{x\in\Z^d}P^x_\mathrm{v}(\shiftsv(A)).
$$
\end{prop}

\subsection{Construction of branching interlacements}

We can construct the model of branching interlacements using either $\big(\Tinfs,\shiftsc,(P^x_\mathrm{c})_{x\in \Z^d}\big)$ or $\big(\Tinfsv,\shiftsv,(P^x_\mathrm{v})_{x\in \Z^d}\big)$. There is no significant difference between them. We use $\big(\Tinfsv,\shiftsv,(P^x_\mathrm{v})_{x\in \Z^d}\big)$ here and will mention the main difference later. Simply write $W_0=\Tinfsv$, $\Shift=\shiftsv$, and $\Shift_k=\Shift^{(k)}$, for $k\in \Z$.

For any $w\in W_0$, we can view $w$ as a function: $\Z\rightarrow \Z^d$ such that
$$
w(k)=\mathrm{the~image~of~the~root~of~}\Shift_k(w).
$$

\begin{remark}
We now see the advantage of introducing $\Shift$: $\Shift$ plays as a time-translation and $w$ is like a two-sided random walk path (though not Markovian). Then we can follow the construction of random interlacements or the branching interlacements for the geometric case. In face, if we let $\mu$ be the degenerate probability measure or the critical geometric probability measure, we will recover random interlacements introduced in \cite{S10V} (modifications needed for the sample space) or the branching interlacements for the geometric distribution in \cite{ARZ15}.
\end{remark}

It is easy to see that (e.g.\ see Proposition~3.4 in \cite{Z162}), for any $K\ssubset \Z^d$ (when $d\geq 5$ as assumed), the set $\{n\in\Z: w(n)\in K\}$ is almost surely finite (under $P^\mathrm{v}_x$ or $P^\mathrm{c}_x$). Hence, we can concentrate only on those `transient' trajectories and let
$$
W=\{w\in W_0:\lim_{|n|\rightarrow\infty}w(n)=\infty\}.
$$

Define the set of spatial trees modulo time-shift by $\overline{W}=W/\cong$, where $\cong$ is the equivalence relation
$$
w_1\cong w_2, \mathrm{~if~} \Shift_k(w_1)=w_2 \mathrm{~for~some~} k\in\Z.
$$
Denote the canonical projection by $\pi:W\rightarrow \overline{W}$ which sends each element in W to its equivalence class in $\overline{W}$.

For any $K\ssubset \Z^d$, let
\begin{equation}\label{def-WA}
W_K=\{w\in W: w(n)\in K, \mathrm{~for~some~}n\in\Z\}, \quad \overline{W}_K=\pi(W_K).
\end{equation}
For any $w\in W_K$, define the `entrance time' by
$$
H_K(w)=\inf\{n:w(n)\in K\}.
$$
Since $w$ is `transient', $H_K(w)\in(-\infty,\infty)$ when $w\in W_K$. We can partition $W_K$ according to the entrance time.
$$
W_K=\bigcup_{n\in \Z}W_K^n, \mathrm{~where~}W_K^n=\{w\in W_K:H_K(w)=n\}.
$$
Define $t_K:W_K\rightarrow W_K^0$ ($\overline{t}_K:\overline{W}_K\rightarrow W_K^0$) by $t_K(w)=w_0$ ($\overline{t}_K(\overline{w})=w_0$), where $w_0$ is the unique element in $W_K^0$ with $w_0\cong w$ ($\pi(w_0)=\overline{w}$).

For every $K\ssubset \Z^d$, we now define a measure $Q_K$ on $W$ by
\begin{equation}\label{def-Q}
Q_K(\bullet)=\sum_{x\in K}P^x_\mathrm{v}(\bullet\cap W_K^0).
\end{equation}

It turns out that for different $K\subseteq K'\ssubset\Z^d$, $Q_K$ and $Q_{K'}$ are consistent in the following sense.
\begin{theorem}\label{pro_consis}
For any measurable event $A\subseteq \overline{W}$, and $K\subseteq K'\ssubset \Z^d$, we have
\begin{equation}
Q_K(\pi ^{-1}(A)\cap W_K)=Q_{K'}(\pi ^{-1}(A)\cap W_K).
\end{equation}
\end{theorem}

Therefore we can define a measure on $\overline{W}$.
\begin{corollary}\label{def-nu}
There exists a unique $\sigma$-finite measure $\nuv$ on $\overline{W}$ which satisfies: for all $K\ssubset\Z^d$,
\begin{equation}\label{def_nu}
\nuv|_{\overline{W}_K}=Q_K\circ \pi^{-1}.
\end{equation}
\end{corollary}

\begin{proof}[Proof of Theorem~\ref{pro_consis}]
Write $B=t_K(\pi ^{-1} (A)\cap W_K)$. Since $Q_K$ ($Q_{K'}$) is supported on $W_K^0$ ($W_{K'}^0$), we have
\begin{equation*}
Q_K(\pi ^{-1}(A)\cap W_K)=Q_K(B)\quad \mathrm{and} \quad Q_{K'}(\pi ^{-1}(A)\cap W_K)=Q_{K'}(t_{K'}(B)).
\end{equation*}
It suffices to show
\begin{equation}\label{equ_QK}
Q_K(B)=Q_{K'}(t_{K'}(B)).
\end{equation}

We partition $W_K^0$ according to the entrance times and the entrance points of $K$ and $K'$. For any $x\in K, y\in K'$ and $n \in \mathbb{Z}^-=\{0,-1,-2,...\}$, write
\begin{equation*}
A_{x,n,y}=\{w\in W:w(0)=x, H_K(w)=0, w(n)=y, H_{K'}(w)=n\}.
\end{equation*}
On $A_{x,n,y}$, $t_{K'}$ is injective, $t_{K'}(w)(\bullet)=w(\bullet +n)$ and
\begin{displaymath}
t_{K'}(A_{x,n,y})=\{w\in W:w(0)=y, H_K(w)=-n, w(-n)=x, H_{K'}(w)=0\}.
\end{displaymath}

Let $B_{x,n,y}=B\cap A_{x,n,y}$. Then $B$ has a countable partition.
\begin{equation*}
B=\bigcup _{x\in K, y\in K',n \in \mathbb{Z}^-}B_{x,n,y}.
\end{equation*}
In order to show \eqref{equ_QK}, it suffices to show
\begin{equation*}
Q_K(B_{x,n,y})=Q_{K'}(t_{K'}(B_{x,n,y})).
\end{equation*}

By \eqref{def-Q}, one can get
\begin{align*}
Q_K(B_{x,n,y})=&P^x_\mathrm{v}[w(n)=y, H_K(w)=0, H_{K'}(w)=n, w(0)=x, w\in\pi^{-1}(A)]\\
=&P^y_\mathrm{v}[w(-n)=x, H_K(w)=-n, H_{K'}(w)=0, w(0)=y,w\in\pi^{-1}(A)]\\
=&Q_{K'}(t_{K'}(B_{x,n,y})),
\end{align*}
where for the second line we use the fact that $(P^x_\mathrm{v})_{x\in\Z^d}$ is invariant under $\shiftv$, see Proposition~\ref{invariance}.
\end{proof}

We have constructed the measure $\nuv$ which will be used as the intensity measure in the construction of the branching interlacement Poisson point process. Before moving to the Poisson point process, let us calculate $\nuv(\overline{W}_K)$.
$$
\nuv(\overline{W}_K)=Q_K(W_K)=Q_K(W^0_K)=\sum_{x\in K}P^x_v(W_K^0).
$$
Note that $P^x_v(W_K^0)$ is just the escape probability $\Es_K(x)$ introduced by the author in \cite{Z161}. $BCap(K)$ is also introduced there as the sum of $\Es_K(x)$ over all $x\in K$. Hence, we have

\begin{prop}
\begin{equation}\label{nuBCap}
\nuv(\overline{W}_K)=\BCap(K).
\end{equation}
\end{prop}

\begin{remark}
If we use the `last visiting time' $D_K(w)=\sup\{n:w(n)\in K\}$ instead of the `first visiting time' $H_K(w)$, then we get another equivalent definition of $\BCap(K)$: $\BCap(K)=\sum_{x\in K}\Esc_K(x)$. See \cite{Z161} for the definition of $\Esc_K(x)$. Note that in \cite{Z161}, $\sum_{x\in K}\Es_K(x)=\sum_{x\in K}\Esc_K(x)$ is proved by establishing the exact asymptotics of the visiting probability of $K$ and for general graphs, this method may fail since the visiting probability may behave quite differently. One can get a similar flavor from the case of random walk and random interlacements.
\end{remark}

We now introduce the space of locally finite point measures on $\overline{W}$.
\begin{multline}\label{def_eq_Omega}
\Omega = \Big\{  \omega = \sum_{n \in I} \delta_{\overline{w}_n}, \mathrm{~where~} \overline{w}_n \in
 \overline{W},\; I\subseteq \N  \\
\mathrm{~and~} \omega(\overline{W}_K ) < \infty
 \mathrm{~for~any~} K \ssubset \Z^d \Big\}.
\end{multline}
For any $u\in[0,\infty)$, let $\P^u$ be the law of a Poisson point process on $\Omega$ with intensity measure $u\nuv$. We can now define the branching interlacement at level $u$
\begin{definition}
The branching interlacement at level $u$ is defined to be the random subset of $\Z^d$ given by
\begin{equation}
\cI^u=\cI^u(\omega) =
\bigcup_{n\geq 0} \mathrm{Range}(\overline{w}_n), \quad \mathrm{where } \quad
 \omega= \sum_{n \ge 0} \delta_{\overline{w}_n}\; \mathrm{~has~law~} \P^u,
\end{equation}
and for $\overline{w}\in \overline{W}$, $\mathrm{Range}(\overline{w})=w(\Z)$, such that $w\in W$ is any element satisfying $\pi(w)=\overline{w}$.
The vacant set of branching interlacement at level $u$ is defined by
\begin{equation}
\cV^u=\cV^u (\omega) = \Z^d \setminus \mathcal{I}^u(\omega).
\end{equation}
\end{definition}

\begin{prop}
For any $u\geq 0$ and $K\ssubset \Z^d$, we have
\begin{equation}\label{VBCap}
P[K\subseteq \cV^u(\omega)]=P[K\cap \cI^u(\omega)=\emptyset]=\exp(-u\BCap (K)).
\end{equation}
\end{prop}
\begin{proof}
$$
P[K\subseteq \cV^u(\omega)]=P[K\cap \cI^u(\omega)=\emptyset]=\exp(-u\cdot\nuv(\overline{W}_K))\stackrel{\eqref{nuBCap}}{=}\exp(-u\BCap (K)).
$$
\end{proof}

\begin{remark}\label{rm1}
\begin{enumerate}
\item Using the inclusion-exclusion principle, one can see that \eqref{VBCap} uniquely determines the laws of $\cI^u$ and $\cV^u$. Precisely, for any $A\subseteq K\ssubset \Z^d$, we have,
\begin{align*}
   P[\cI^u\cup K=A]=&
   \sum_{B\subseteq A}(-1)^{|A\setminus B|}P[\cI^u\cup (K\setminus B)=\emptyset]\\
   =&\sum_{B\subseteq A}(-1)^{|A\setminus B|}\exp(-u\BCap(K\setminus B)).
\end{align*}

\item As mentioned earlier, if we let $W=\Tinfs,\Shift=\shiftsc$ and replace $P^x_\mathrm{v}$ by $P^x_\mathrm{c}$ in \eqref{def-Q}, we can establish another measure $\nuc$ on $\overline{W}$. It is elementary to see that $\nuc(\overline{W}_K)=\BCap(K)/2$, for any $K\ssubset \Z^d$. Therefore, we can use $2u\nuc$ for the intensity measure in the construction of the branching interlacement set at level $u$ without changing its law.
\item Obviously $\nuv$ ($\nuc$) is invariant under spatial translation: $\overline{w}\rightarrow \overline{w}+x$, for every $x\in\Z^d$. Hence, $\cI^u$ is space-translation invariant (i.e.\ $\cI^u\stackrel{d}{=}\cI^u+x$ for every $x\in\Z^d$). On the other hand, $\nuc$ is invariant under the time inversion while $\nuv$ is not.
\item Since $\cI^u$ consists of infinite tree-indexed random walk trajectories, $\cV^u$ possesses the following screening effect: if $K'$ separates $K$ from infinity (with respect to the connectedness induced by $\theta$), then
$$
\{K'\subseteq \cV^u\}\subseteq \{K\subseteq \cV^u\}.
$$
\item In \cite{Z161}, it is showed that the branching capacity of the cube (or the ball) $B_0(r)$ with radius $r$ behaves like
$$
\BCap(B_0(r))\asymp r^{d-4}.
$$
Hence, for some positive constants $c,C$,
$$
\exp(-Cur^{d-4})\leq P[B_0(r)\subseteq \cV^u]\leq\exp(-cur^{d-4}).
$$
In particular, there is generally not exponential decay with $|A|$ of $P[A\subseteq \cV^u]$.
\item Since $\BCap(K\cup K')\leq \BCap(K)+\BCap(K')$, one can see that the events $\{K\subseteq \cV^u\}$ and $\{K'\subseteq \cV^u\}$ are positively correlated:
$$
P[\{K\cup K'\subseteq \cV^u\}]\geq P[\{K\subseteq \cV^u\}]P[\{K'\subseteq \cV^u\}].
$$
On the other hand, the FKG inequality holds for $\cV^u$. More precisely, for every pair of increasing random variables $X,Y$ with finite second moment with respect to the law of $\cV^u$, denoted by $\mathcal{Q}^u$, we have
$$
\int XY \mathrm{d}\mathcal{Q}^u\geq \int X \mathrm{d}\mathcal{Q}^u \int Y \mathrm{d}\mathcal{Q}^u.
$$
This can be proved in much the same way as in \cite{T09I}
\item Using the results in \cite{Z161}, it is not difficult to get
$$
\BCap(\{x\})+\BCap(\{y\})-\BCap(\{x,y\})\sim \frac{c}{\|x-y\|^{d-4}},
$$
where $\|x\|=\sqrt{x\cdot Q^{-1}x}/\sqrt{d}$, with $Q$ being the covariance matrix of $\theta$. From this, one can see
$$
\quad\quad\quad\mathrm{Cov}(1_{x\in \cI^u},1_{y\in \cI^u})=\mathrm{Cov}(1_{x\in \cV^u},1_{y\in \cV^u})\sim \frac{cu}{\|x-y\|^{d-4}} ,\mathrm{~as~}\|x-y\|\rightarrow \infty.
$$
\item The constant ``$4$'' in the exponent comes from the fact that tree-indexed random walks (with regularity assumptions) behave like $4$-dimensional subjects (while random walks behave like $2$-dimensional subjects). This fact is reflected in many results on tree-indexed random walk, see e.g. \cite{L05R,L99S,LL16T,PZ16C,Z161,Z162,Z163,ZThesis}.
\item The measure $ \nu_\mathrm{v}|_{\overline{W}_K}\circ\bar{t}^{-1}_K$ on $W_K^0$ is
$$
\sum_{x\in K}\Es_K(x)P^\mathrm{v}_x[\bullet|H_K=0]=\BCap(K)\sum_{x\in K}\frac{\Es_K(x)}{\BCap(K)}P^\mathrm{v}_x[\bullet|H_K=0].
$$
Note that the sum on the right hand site is a probability law (on the space of tree-indexed random walks).  It is not difficult to verify that the entering measure (of $K$) by a tree-indexed random walk with this law coincides with the `branching harmonic measure' of $K$, $m_K$, introduced in Section 9 in \cite{Z161}. From this, one can obtain Proposition~\ref{BI-mK}.
\item The constructions here for the random walk in $\Z^d,d\geq5$ can be straightforwardly generalized to the case of an infinite locally finite connected graph $G=(V,\mathcal{E})$ with vertex set $V$ and (undirected) edge set $\mathcal{E}$, endowed with positive weights $\lambda (e)>0, e\in E$, such that the corresponding nearest neighbor walk on $V$ with transition probability
$$
p(x,y)=\frac{\lambda(\{x,y\})}{\sum_{z:\{x,z\}\in \mathcal{E}}\lambda(\{x,z\})}
$$
is `branching transient' in the sense that any finite subset of $V$ is visited finitely many times by the tree-indexed random walk with law $P^\mathrm{v}_x$, almost surely. Note that the corresponding random walk on $V$ is reversible and has a stationary measure $\lambda(x)=\sum_{z:\{x,z\}\in \mathcal{E}}\lambda(\{x,z\})$. In this set-up some of our definitions need to be slightly modified. For example, one should insert multipliers $\lambda(x)$ into the sum in \eqref{def-Q}.

\end{enumerate}
\end{remark}

\section{Local limits of random trees and decomposing large random trees into small ones}

In this section, we consider random trees (GW-trees) conditioned on size. Recall that we fix the offspring distribution $\mu$ and always assume that $\mu$ is critical (with mean one) and with finite variance $\sigma^2>0$. For simplicity, we assume further that $\mathrm{span}(\mu)=1$ and leave the minor modifications when $\mathrm{span}(\mu)>1$ to the reader. To make our arguments smooth, we concentrate mainly on the case when $\mu$ has finite exponential moments i.e.\ $\sum_{i\in\N}c^i\mu(i)<\infty$, for some $c>1$. However, for the local limit result of random trees, we do not need this assumption. We write $\Tbf$ for an unconditioned GW-tree and $\Tfini$ for the sample space, i.e.\ the set of finite rooted order trees. We define $\Tbf^n$, called the GW-tree with size $n$, as $\Tbf$ conditioned on $|\Tbf|=n$ (recall that, for a tree $\tbf$, we also write $\tbf$ for its vertex set). For any $\tbf\in\Tfini$, we adopt the depth-first search order (from the root) on $\tbf$. We write $V_i(\tbf)$ (even $V_i$ when $\tbf$ is obvious) for the $i$-th vertex of $\tbf$ starting from
$V_0(\tbf)=o$, the root of $\tbf$. The main tool we will use is the so-called Lukasiewisz walks (L-walks). For any $\tbf\in\Tfini$ with $|\tbf|=n$, its L-walk $\mathbf{L}_\tbf:\{0,1,\dots,n\}\rightarrow \Z^d$ is defined by $\mathbf{L}_\tbf(0)=0$ and $\mathbf{L}_\tbf(k)=\sum_{i=0}^{k-1}(\dout(V_i(\tbf))-1)$, for $k\in[[1, n]]$. It is easy to check that the L-walk of the GW-tree $\Tbf$ is distributed as a random walk on $\Z$ with jump distribution $\mu_{-1}$ given by $\mu_{-1}(j)=\mu(j+1)$ for every $j\geq-1$, which starts at $0$ and stops at the fist visiting time of $-1$. In particular, $|\Tbf|$ has the same distribution as the fist visiting time of such a random walk.

\subsection{Some estimates on L-walks}
In this subsection we collect useful lemmas about L-walks.

For a sequence $\mathbf{x}=(x_1,\dots,x_n)$, the walk with step $\mathbf{x}$ is just the partial sum sequence $s_0=0,s_1,\dots,s_n$ where $s_j=\sum_{i=1}^j x_i$. Say the walk fist visits $b$ at time $k$ if $s_i\neq b$ for $i<k$ and $s_k=b$. For $i\in\{1,\dots,n\}$, the $i$-th cyclic shift of $\mathbf{x}$, denoted by $\mathbf{x}^{(i)}$, is the sequence of length $n$ whose $j$-th term is $x_{i+j}$ with $x_{i+j}=x_{i+j-n}$ for $i+j>n$.
\begin{lemma}\label{Kemperman}
Let $\mathbf{x}=(x_1,\dots,x_n)$ be a sequence with values in $\Z$ and sum $-k$ for some $1\leq k\leq n$. Then there are at most $k$ distinct $i\in\{1,\dots,n\}$ such that the walk with step $\mathbf{x}^{(i)}$ fist visits $-k$ at time $n$. Moreover, if the values of $x_i$ are in $\{-1,0,1,2,\dots\}$, then there are exactly $k$ distinct such $i\in\{1,\dots,n\}$.
\end{lemma}
\begin{proof}
When $x_i\in \{-1,0,1,\dots\}$, it follows from Kemperman's formula (see Lemma~6.1 \cite{P06C}). For the general case, one can replace each $x_i=-j\in\{-2,-3,\dots\}$ by $j$ successive $-1$. Then the new sequence has exactly $k$ cyclic shifts which first visit $-k$ at the end. Therefore the original sequence has at most $k$ cyclic shifts which first visit $-k$ at the end.
\end{proof}

Let $L(n)=\sum_{1}^n X_i$ where $X_i$ are i.i.d. with distribution $\mu_{-1}$. Set $H=H_{-1}=\inf\{k\geq0, L(k)=-1\}$ the fist visiting time of $-1$. By the lemma above, we get
$$
P(|\Tbf|=n)=P(H=n)=\frac{1}{n}P(L(n)=-1).
$$
If $(\Tbf_i)_{i\in \N^+}$ are independent copies of $\Tbf$, by concatenating their L-walks, one can get
\begin{equation}\label{eq-Kem}
  P(\sum_1^m |\Tgw_i|=n)=P(H_{-m}=n)=\frac{m}{n}P(L_n=-m),
\end{equation}
where $H_{-m}$ is the fist visiting time of $-m$.

We also need standard estimates for random walks. For $x\in\Z, n\in\N^+$, write
$$
\overline{p}_n(x)=\frac{1}{\sqrt{2\pi}\sigma}\frac{1}{\sqrt{n}}\exp(-\frac{x^2}{2n\sigma^2}).
$$

\begin{lemma}\label{lem-rw}
(See Theorem~2.3.9, Theorem~2.3.11 and Corollary~A.2.7 in \cite{LL10R} and Lemma~2.1 in \cite{J06R}) There exists $c$, such that
\begin{align*}
P(L(n)=x)=\overline{p}_n(x)+o(\frac{1}{\sqrt{n}}),\\
 P(L(n)=x)\preceq \frac{1}{\sqrt{n}}\exp(-\frac{cx^2}{n}),\;\mathrm{when~}x\leq 0.
\end{align*}
Moreover, if $\mu$ has finite exponential moments, then for some $c_1,c_2$,
\begin{eqnarray*}
P[L(n)=x]=\overline{p}_n(x)\exp\left(O(\frac{1}{\sqrt{n}}+\frac{|x|^3}{n^2})\right), \;\mathrm{when~}|x|\leq c_1n,\\
P[\max_{0\leq j\leq n}L(j)\geq r\sqrt{n}]\leq \exp(-c_2r^2),\; \mathrm{when~}r\in[0,c_1\sqrt{n}].
\end{eqnarray*}
\end{lemma}

By this lemma and \eqref{eq-Kem}, we get
\begin{equation}\label{est-tree-size1}
P[|\Tgw|=n]\sim\frac{1}{\sqrt{2\pi}\sigma n^{3/2}},
\end{equation}
\begin{equation}\label{est-tree-size2}
P[\sum_1^m |\Tgw_i|=n]\preceq \frac{m}{n^{3/2}} \exp(-\frac{cm^2}{n}),
\end{equation}
\begin{equation}\label{est-tree-size3}
P[\sum_1^m |\Tbf_i|=n]  \sim\frac{1}{\sqrt{2\pi}\sigma } \frac{m}{n^{3/2}}\exp(-\frac{m^2}{2n\sigma^2}), \;  \mathrm{when~ }m=O(\sqrt{n}).
\end{equation}

Since we are interested in GW-trees with given sizes, from the view of L-walks, we need to study the random walk $L=(L(i))_{i\in\N}$ conditioned on $H_{-1}=n$. Write $P^n$ for the law of $L$ conditioned on $H_{-1}=n$.

\begin{lemma}\label{ineq-L}
Assume that $\mu$ has finite exponential moments. Then, there exists $c>0$ such that, for $m\in[[1,n-1]]$ and $h\in [[0,cA]]$ with $A=m\wedge(n-m)$,
\begin{align*}
&P^n[L(m)=h]\preceq \frac{(h+1)^2}{A^{3/2}}\exp (-c\frac{(h+1)^2}{A}),\\
&P^n[L(m)\leq h]\preceq (\frac{h+1}{\sqrt{A}})^3,\\
&P^n[L(m)\geq h]\preceq \exp(-c\frac{(h+1)^2}{A}).
\end{align*}
\end{lemma}
\begin{proof}
\begin{align*}
P^n(L(m)=h)&=\frac{P[H=n,L(m)=h]}{P[H=n]}\\
=&\frac{P\left[L(m)=h;L(i)\geq 0, i\in[[0,m]]\,\right]P_h[H =n-m]}{P[H=n]},
\end{align*}
where $P_h$ is the law of $L$ when staring at $h$
and we use the Markov property in the last line.

By the equations after Lemma~\ref{Kemperman}, we have
\begin{align*}
P_h[H =n-m]&=P(H_{-h-1}=n-m)\\
=&\frac{h+1}{n-m}P(L(n-m)=-h-1)\preceq \frac{h+1}{(n-m)^{3/2}}\exp(-c\frac{(h+1)^2}{n-m}),\\
&\quad\quad\quad P[H=n]\sim\frac{1}{n^{3/2}}.
\end{align*}
For the other term, by inserting one up-step and then reversing the walk, one can see (for any fixed $j$ with $\mu_{-1}(j)>0$)
\begin{align*}
&P\left[L(m)=h;L(i)\geq 0, i\in[[0,m]]\,\right]\\
\leq &P \left[L(m+1)=h+j;L(i)\geq j, i\in[[1,m+1]]\,\right]/\mu_{-1}(j)\\
\leq&P_{h+j}\left[L^-(m+1)=0;L^-(i)\geq 1, i\in[[0,m]]\,\right]/\mu_{-1}(j)\\
\preceq&\frac{h+j}{m+1}P(L^-(m+1)=-h-j)\preceq\frac{h+1}{m^{3/2}}\exp\Big(-c\frac{(h+1)^2}{m}\Big),
\end{align*}
where $L^-$ is the reversed random walk (i.e.\ with jump distribution $\mu_{-1}^-(\bullet)\doteq\mu_{-1}(-\bullet)$) and in the last line we apply Lemma \ref{Kemperman} and Lemma~\ref{lem-rw}.

From what have been proved, it follows that
\begin{align*}
  P^n(L(m)=h)&\preceq \frac{\frac{(h+1)^2}{(m(n-m))^{3/2}}\exp\big(-c(\frac{(h+1)^2}{m}+\frac{(h+1)^2}{n-m})\big)}{n^{-3/2}}\\
  &\preceq \frac{(h+1)^2}{A^{3/2}}\exp \Big(-c\frac{(h+1)^2}{A}\Big).
\end{align*}
This completes the proof of the first assertion. By summation and the last lemma, one can get the other assertions.
\end{proof}

\begin{lemma}\label{lem-Lw-leaf}
Let $k=k(n)\in [[1,n-1]]$ with $A=A(n)=k(n)\wedge (n-k(n))\rightarrow \infty$. Then
$$
\lim_{n\rightarrow\infty}P(V_k \mathrm{~in~}\Tbf^n\mathrm{~is~a~leaf})= \mu(0).
$$
Moreover, if $\mu$ has finite exponential moments, then
$$
\sup_{n,k\in\N^+:A<k<n-A}P(V_k \mathrm{~in~}\Tbf^n\mathrm{~is~a~leaf})=\mu(0)\Big(1+O(\frac{1}{A^{3/8}})\Big).
$$
\end{lemma}
\begin{proof}
It is standard that for any $\epsilon>0$, we can find $\delta\in(0,1/2)$, such that (when $A$ is sufficiently large)
$$
P^n(L(k)\in[\delta\sqrt{A},\frac{1}{\delta}\sqrt{A}]^\mathrm{c})<\epsilon.
$$
Conditionally on the event $\mathbf{L}_{\Tbf^n}(k)=h$, we have (when $h\in[[\delta\sqrt{A},\delta^{-1}\sqrt{A}]]$)
\begin{align*}
P&[V_k \mathrm{~in~}\Tbf^n\mathrm{~is~a~leaf}|\mathbf{L}_{\Tbf^n}(k)=h]=\frac{P[H=n,L(k)=h,L(k+1)=h-1]}{P[H=n,L(k)=h]}\\
=&\frac{\mu(0)P_{h-1}[H=n-k-1]}{P_{h}[H=n-k]}=\frac{\mu(0)\frac{h}{n-k-1}P(L(n-k-1)=-h)}{\frac{h+1}{n-k}P(L(n-k)=-h-1)}\rightarrow \mu(0),
\end{align*}
where we use the Markov property and Lemma~\ref{lem-rw}. The proof of the first assertion is complete.

When $\mu$ has finite exponential moments. By the last lemma, we can find large $C$, such that
$$
P^n[L(k)\in[A^{3/8}, C\sqrt{A\log A}]^\mathrm{c}]\preceq \frac{1}{A^{3/8}}.
$$
Similarly, we have (when $h\in[[A^{3/8}, C\sqrt{A\log A}]]$)
\begin{align*}
P&(V_k \mathrm{~in~}\Tbf^n\mathrm{~is~a~leaf}|L_{\Tbf^n}(k)=h)=\frac{\mu(0)\frac{h}{n-k-1}P(L(n-k-1)=-h)}{\frac{h+1}{n-k}P(L(n-k)=-h-1)}\\
&=\mu(0)\frac{h}{h+1}\frac{n-k}{n-k-1}\frac{\overline{p}_{n-k-1}(-h)(1+O(1/(n-k-1)^{0.4}))}{\overline{p}_{n-k}(-h-1)(1+O(1/(n-k)^{0.4}))}=\mu(0)(1+O(\frac{1}{A^{3/8}})).
\end{align*}
We finish the proof of the second assertion.
\end{proof}

\subsection{Local limits of large Galton-Watson trees at a prefixed vertex}

Let us introduce an unrooted random tree $\Tl$, which appears in \cite{A91A,S18L}. Intuitively, $\Tl$ is $\Tbf^\mathrm{v}$ re-rooted at infinity. Formally, $\Tl$ is defined as follows. Start with a semi-infinite path $u_0,u_1,\dots$. Let $u_0$ be the root of an independent copy of the unconditioned GW-tree. For each $i\geq 1$, $u_i$ produces, independently, a random number of children according to $\musb$, the size-biased probability measure of $\mu$. Let $u_{i-1}$ be identified with a uniformly chosen child of $u_i$ and all other children of $u_i$ become the root of an independent copy of the unconditioned GW-tree.

We follow the notations used in \cite{S18L}. Let  $\tbf\in\T_\mathrm{f}$, $v\in \tbf$, and $l\in \N$. When the vertex $v$ has an $l$-th ancestor $v_l$, we define the pointed plane tree $H_l(\tbf,v)$ as $[\tbf]_{v_l}$ (we write $[\tbf]_v$ for the subtree of $\tbf$ generated by all descendants of $v$, rooted at $v$), pointed at $v$. Otherwise, when the vertex $v$ has no $l$-th ancestor, we simply set $H_l(\tbf,v)=\diamond$ for some inessential symbol $\diamond$. Let $H_l(\Tl,u_0)$ be the subtree of $\Tl$ generated by all descendants of $u_l$, rooted at $u_l$ and pointed at $u_0$.

We now state our theorem.
\begin{theorem}\label{thm-ll}
For any $k=k(n)\in[[1,n/2]]$ and $l=l(n)\in[[0,k(n)]]$ satisfying $\lim_{n\rightarrow \infty}k(n)=\infty$ and $\lim_{n\rightarrow \infty}k(n)/ (l(n)+1)^{2}=\infty$. Then we have
$$
\lim_{n\rightarrow \infty}\sup_{m\in[[k(n),n-k(n)]]}
\dTV(H_{l}\left(\Tbf^n,V_m),H_{l}(\Tl,u_0)\right)=0,
$$
where $\dTV(\bullet,\bullet)$ is the total variance distance.
Moreover, if $\mu$ has finite exponential moments, then we have the following bounds for the error term:
\begin{align*}
    \sup_{m\in[[k(n),n-k(n)]]}&\dTV(H_{l}\left(\Tbf^n,V_m),H_{l}(\Tl,u_0)\right)\\
    \preceq&\left\{\begin{array}{cc}
        \frac{1}{k^{3/8}}+\left(\frac{l+1}{\sqrt{n}}\right)^{2/3}, & \mathrm{~when~}((l+1)n)^{2/3}\leq k; \\
        \frac{l+1}{\frac{1}{k^{3/8}}+\sqrt{k}}+\frac{k}{n}, & \mathrm{~when~}((l+1)n)^{2/3}\geq k.\\
    \end{array}\right.
\end{align*}

\end{theorem}

We give two lemmas before proving the theorem.
\begin{lemma}\label{lem-ll1}
$$
P(|H_{l}(\Tl,u_0)|\geq N)\preceq\frac{l+1}{\sqrt{N}}.
$$
\end{lemma}
\begin{proof}
Note that $|H_{l}(\Tl,u_0)|$ can be written as the sum of $|H_{0}(\Tl,u_0)|$ and the independent differences
$$
|H_{i}(\Tl,u_0)|-|H_{i-1}(\Tl,u_0)|\stackrel{d}{=}S_{\hat{\xi}-1}, i=1,\dots,l;
$$
where $\hat{\xi}$ has the size-biased distribution $\musb$ and
$$
S_m=\xi_1+\dots+\xi_m+1
$$
with $(\xi_j)_{j\in\N}$ being independent copies of $|\Tbf|$. Hence we can bound $|H_{l}(\Tl,u_0)|$ stochastically by $S_{M_l}$ with
$$
M_l=\hat{\xi}_1+\dots+\hat{\xi}_{l+1},
$$
the sum of $l+1$ independent copies of $\hat{\xi}$. By \eqref{est-tree-size2},
$$
P(S_m\geq n)\preceq m/\sqrt{n}.
$$
Hence, we get
$$
P(S_{M_l}\geq N)\leq \sum_{i}P(M_l=i)P(S_i\geq N)\preceq \sum_{i}P(M_l=i)\frac{i}{\sqrt{N}}=\frac{(l+1)E[\hat{\xi}]}{\sqrt{N}}.
$$
\end{proof}

For $\tbf\in\Tfini$ and $v\in\tbf$, we simply write $(\tbf,v)$ for the pointed tree $\tbf$ pointed at $v$. Let $\mathcal{E}_{k,l}$ denote the set of pointed trees $(\tbf, v)$ of a plane tree $\tbf$ having at most $k$ vertices and a vertex $v$ with height $l$, such that $P(H_{l}(\Tl,u_0)=(\tbf,v))>0$. For such $(\tbf,v_0)\in\mathcal{E}_{k,l}$, it is elementary to see that
$$
P[H_{l}(\Tl,u_0)=(\tbf,v_0)]=\prod_{v\in\tbf}\mu(\dbf_o (v)).
$$
\begin{lemma}\label{lem-ll2}
For any $(\tbf,v)\in\mathcal{E}_{k,l}$ $n\geq|\tbf|$ and $m\geq a$ with $a$ being the number of vertices in $\tbf$ that are before $v$, we have
$$
\frac{P[H_{l}(\Tbf^n,V_m)=(\tbf,v)]}{P[H_{l}(\Tl,u_0)=(\tbf,v)]}=\frac{P(V_{m-a} \mathrm{~in~} \Tbf^{n-|\tbf|+1}\mathrm{~is~a~leaf})}{\mu(0)}\cdot\frac{P(|\Tbf|=n-|\tbf|+1)}{P(|\Tbf|=n)}.
$$
\end{lemma}

\begin{proof}
Note that the out-degree of the $i$-th vertex $V_i$ in $\Tbf$ is uniquely determined by the difference of the L-walk values at $i$ and $i+1$. Hence $H_{l}(\Tbf^n,V_m)=(\tbf,v)$ if and only if the translated L-walk of $\Tbf^n$ restricted on $[m-a,m+|\tbf|-a]$ equals to the L-walk of $\tbf$, i.e.\
\begin{equation}
L_{\Tbf^n}(i+m-a)-L_{\Tbf^n}(m-a)=L_\tbf(i), \qquad i=0,1,\dots,|\tbf|.
\end{equation}

As before, write $L(i)=X_1+\dots+X_i$ for the random walk with jump distribution $\mumo$. Then
$$
P[H_{l}(\Tbf^n,V_m)=(\tbf,v)]=\frac{P[L(1),\dots,L(n-1)\geq0,L(n)=-1; (\star)]}{P[|\Tbf|=n]},
$$
where we write $(\star)$ for
\begin{equation*}
L(i+m-a)-L(m-a)=L_\tbf(i), \qquad i=0,1,\dots,|\tbf|.
\end{equation*}
By conditioning on the value of $L(m-a)$, we can rewrite the numerator as follows.
\begin{align*}
&\sum_{x}P[L(i)\geq0,i\in[[0,m-a]];L(m-a)=x]\times\\
P&[(\star); L(j)\geq 0,j\in[[m-a,n]];L(n)=-1|L(m-a)=x]\\
=&\sum_{x}P[L(i)\geq0,i\in[[0,m-a]];L(m-a)=x]\times\\
P&[(\star);L(j)\geq0,j\in[[m+|\tbf|-a,n]];L(n)=-1|L(m-a)=x]\\
=&\sum_{x}P[L(i)\geq0,i\in[[0,m-a]];L(m-a)=x]\times P[(\star)|L(m-a)=x]\times\\
P&[L(j)\geq0,j\in[[m+|\tbf|-a,n]];L(n)=-1|L(m+|\tbf|-a)=x-1],
\end{align*}
where we use the Markov property twice and note that $(\star)$ together with $L(m-a)=x$ implies $L(|\tbf|+m-a)=x-1$ (and hence $x\geq 1$).
Also by the Markov property, we obtain that
\begin{equation*}\label{eq2.1}
P[(\star)|L(m-a)=x]=\prod_{v\in\tbf}\mu(\dout(v))=P[H_{l}(\Tl,u_0)=(\tbf,v)].
\end{equation*}

We now turn to the other terms. By inserting an extra down-step between them (which corresponds to a factor $\mu(0)$), one can connect the L-walks on $[0,m-a]$ and $[|\tbf|+m-a,n]$ . From this, we get
\begin{align*}
 &\sum_{x}P[L(i)\geq0,i\in[[0,m-a]];L(m-a)=x]\times \\
&\quad P[L(j)\geq0,j\in[[m+|\tbf|-a,n]];L(n)=-1|L(m+|\tbf|-a)=x-1]\\
=&P[L(i)\geq0,j\in[[0,n-|\tbf|+1]];\\
&\quad\quad\quad\quad L(m-a+1)-L(m-a)=-1,L(n-|\tbf|+1)=-1]/\mu(0)\\
=&\frac{P[V_{m-a} \mathrm{~in~} \Tbf^{n-|\tbf|+1}\mathrm{~is~a~leaf}]\, P[|\Tbf|=n-|\tbf|+1]}{\mu(0)}.
\end{align*}
By rearranging all equations obtained, we finish the proof.
\end{proof}

\begin{proof}[Proof of Theorem \ref{thm-ll}]
To simplify notations, we let $\Pl_l(\sbr)=P[H_{l}(\Tl,u_0)=\sbr]$, $\Pl^n_{l,m}(\sbr)=P[H_{l}(\Tbf^n,V_m)=\sbr]$ for any pointed plane tree $\sbr \in A_l=\{\sbr:\Pl_l(\sbr)>0\}$ and $\dTV(\Pl_l,\Pl^n_{l,m})=\dTV(H_{l}(\Tl,u_0),H_{l}\left(\Tbf^n,V_m)\right)$.

From Lemma \ref{lem-Lw-leaf}, Lemma \ref{lem-ll2} and Lemma \ref{lem-rw}, we have (when $|\sbr|/k\rightarrow 0$)
\begin{equation}\label{cc}
  \Pl^n_{l,m}(\sbr)=\Pl_l(\sbr)(1+o(1)).
\end{equation}
Hence,
\begin{align*}
\dTV(\Pl_l,\Pl^n_{l,m})&\leq\sum_{\sbr\in A_l:|\sbr|>\epsilon k}\Pl_l(\sbr)+\sum_{\sbr\in A_l:|\sbr|\leq \epsilon k}|\Pl_l(\sbr)-\Pl_{l,m}^n(\sbr)|\\
\preceq&\frac{l+1}{\sqrt{\epsilon k}}+\sum_{\sbr\in A_l:|\sbr|\leq \epsilon k}\Pl_1(\sbr)|o(1)|.
\end{align*}
We can choose $\epsilon$ small enough to make the second term as small as we want and then let $n\rightarrow \infty$ to make the first term also small. This completes the proof of first assertion. When $\mu$ has finite exponential moments, by Lemma~\ref{lem-Lw-leaf}, Lemma~\ref{lem-ll2} and \eqref{est-tree-size1}, one can get(when $|\sbr|\leq x/2$ and $x\leq k$)
$$
\Pl^n_{l,m}(\sbr)=\Pl_l(\sbr)(1+O(\frac{1}{k^{3/8}})+O(\frac{x}{n})).
$$
\begin{itemize}
    \item Case I: $((l+1)n)^{2/3}\leq k$. Let $x=((l+1)n)^{2/3}$. we have
    \begin{align*}
     \dTV(\Pl_l,\Pl^n_{l,m})&\leq\sum_{\sbr\in A_l:|\sbr|>x/2}\Pl_l(\sbr)+\sum_{\sbr\in A_l:|\sbr|\leq x/2}|\Pl_l(\sbr)-\Pl_{l,m}^n(\sbr)|\\
     \preceq&\frac{l+1}{\sqrt{x}}+O(\frac{1}{k^{3/8}})+O(\frac{x}{n})=O(\frac{(l+1)^{2/3}}{n^{1/3}})+O(\frac{1}{k^{3/8}}).
    \end{align*}
    \item Case II: $((l+1)n)^{2/3}\geq k$. we have
    \begin{align*}
     \dTV(\Pl_l,\Pl^n_{l,m})&\leq\sum_{\sbr\in A_l:|\sbr|>k/2}\Pl_l(\sbr)+\sum_{\sbr\in A_l:|\sbr|\leq k/2}|\Pl_l(\sbr)-\Pl_{l,m}^n(\sbr)|\\
     \preceq&\frac{l+1}{\sqrt{k}}+O(\frac{1}{k^{3/8}})+O(\frac{k}{n}).
    \end{align*}
\end{itemize}
\end{proof}

\subsection{Decomposing large random trees into small ones}

In this subsection we consider GW-trees conditioned on siza. From now on, we always assume that $\mu$ has exponential moments (in addition to the previous assumptions). Given a $\tbf\in\T_\mathrm{f}$ with size $|\tbf|=n$, an interval $[[a,b]]\subset[0,n]$ (with $a,b\in\N$) and a non-root vertex $v\in\tbf$, we say $v$ is an M-vertex in $[[a,b]]$ (of $\tbf$) and $[\tbf]_v$ is an M-subtree in $[[a,b]]$ (of $\tbf$) , if
\begin{itemize}
    \item all vertices of $[\tbf]_v$ are in $\{V_a,\dots,V_b\}$;
    \item $v$ is not in the unique path connecting $V_a$ and $V_b$;
    \item the parent of $v$ is in the unique path connecting $V_a$ and $V_b$.
\end{itemize}

We now state a lemma.
\begin{lemma}\label{lem-numofst}
For any $\alpha\in(0,1)$, there exist $C,c_1,c_2$ (depending on $\alpha$), such that: for any $N=N(n)\in[[n^\alpha,n]]$ and $m=m(n)\in [C\sqrt{\log (n+1)},c_1\sqrt{N}]$, we have
\begin{align*}
P[\exists a\in[[0,&n-N(n)-1]]:\sharp\{k\in[[a,a+N(n)]]: V_k\mathrm{~is~an~M-vertex~}\\
&\;\mathrm{in~}[[a,a+N(n)]] \mathrm{~of~} \Tbf^n\}\geq m(n)\sqrt{N(n)}]\preceq \exp(-c_2m^2(n));\\
P[\exists a\in[[0,&n-N(n)-1]]:\\
&\quad d(V_a(\Tbf^n),V_{a+N(n)}(\Tbf^n))\geq m(n)\sqrt{N(n)}]\preceq \exp(-c_2m^2(n)),
\end{align*}
where $d(\bullet,\bullet)$ is the graph distance.
\end{lemma}
\begin{proof}
Fix any $a\in[[0,n-N-1]]$. Write $I=[[a,a+N(n)]]$, $v_1=V_a$, $v_2=V_{a+N(n)}$ and $v'$ for the latest common ancestor of $v_1,v_2$. We begin with the first assertion. If the number of M-vertices in $I$ is larger than $m(n)\sqrt{N(n)}$, then either the number of M-vertices attached to the (simple) path connecting $v_1,v'$ is larger than $m(n)\sqrt{N(n)}/2$, or the number of M-vertices attached to the path connecting $v_2,v'$ is larger than $m(n)\sqrt{N(n)}/2$.

For the first case, assume that all M-subtrees attached to the path connecting $v_1,v'$ are $\tbf_1,\dots,\tbf_j$ and  $|\tbf_1|+\dots+|\tbf_j|=J$. Then we have
\begin{equation}\label{cs-eq1}
j\geq m(n)\sqrt{N(n)}/2, J< N(n).
\end{equation}
From the relation between a tree and its L-walk, one can easily see that the translated L-walk of $\Tbf^n$ restricted on $[[a+1,a+J+1]]$, is just the concatenation of L-walks of $\tbf_1,\dots,\tbf_j$. Precisely,
\begin{equation*}
    L_{\Tbf^n}(a+1+k)-L_{\Tbf^n}(a+1)=L_{\tbf_i}(k-|\tbf_1|-\dots-|\tbf_{i-1}|)-(i-1),
\end{equation*}
when $|\tbf_1|+\dots+|\tbf_{i-1}|\leq k\leq |\tbf_1|+\dots+|\tbf_{i}|$. Especially, we have
\begin{equation}\label{L-walk-c2}
\min_{k\in[[a+1,a+N]]}L_{\Tbf^n}(k)-L_{\Tbf^n}(a+1)\leq L_{\Tbf^n}(a+J+1)-L_{\Tbf^n}(a+1)=-j.
\end{equation}

The L-walk of $\Tbf^n$ is a conditioned random walk conditioned on $H_{-1}=n$. we know that
\begin{equation}\label{cs-eq2}
 P(H_{-1}=n)\sim cn^{-3/2}.
\end{equation}
On the other hand, by Lemma~\ref{lem-rw}
$$
P[\min_{k\in[[0,N]]}L(k)\leq -j]\preceq N\exp(-c\frac{j^2}{N}),
$$
where as before, $L$ is the (unconditioned) random walk starting at $0$.  By this and \eqref{cs-eq1}, we get
$$
P(\eqref{L-walk-c2})\preceq N\exp(-cm^2(n)).
$$
Noting that (by choosing $C$ large enough) $\exp(-cm^2(n))$ decays faster than any polynomials of $n$, we get the desired bound for this case.

For the case when the number of M-vertices attached to the path connecting $v_2,v'$ is larger than $m(n)\sqrt{N(n)}/2$, we can reduce it to the first case by considering the L-walk from the right side. For a $\tbf\in \T_\mathrm{f}$, its L-walk from the right side can be obtained as follows. First, for every vertex, reverse the order of its children. Then the L-walk of this tree (with the reversed order) is just the L-walk from the right side of $\tbf$. We leave the details of the reduction to the reader.

We now turn to the second assertion. When $d(v_1,v_2)\geq m(n)\sqrt{N(n)}$, either $d(v_2,v')\geq m(n)\sqrt{N(n)}/2$ or $d(v_1,v')\geq m(n)\sqrt{N(n)}/2$. As in the proof of the first assertion, (the latter case can be reduced to the former one) we can assume that
\begin{equation}\label{L-walk-c3}
    d(v_2,v')\geq m(n)\sqrt{N(n)}/2.
\end{equation}
Consider the translated L-walk of $\Tbf^n$ in $I$, $L':[[0,N(n)]]\rightarrow \Z$ :
$$
L'(i)=L_{\Tbf^n}(a+i)-L_{\Tbf^n}(a).
$$
Note that $L'$ can be regarded as the concatenation of L-walks of GW-trees (though the traverse of the last GW-tree is generally not `finished') and \eqref{L-walk-c3} corresponds to the event that the height of the $N(n)$-th vertex in these GW-trees is at least $ m(n)\sqrt{N(n)}/2$. Since the L-walk of $\Tbf^n$ is a conditioned random walk conditioned on the event $H_{-1}=n$, which has probability of order $n^{-3/2}$, it suffices to show that for an unconditioned random walk $L$ starting at $0$ with distribution $\mu_{-1}$, (which can be regarded as the concatenation of L-walks of independent GW-trees), the probability of the event that the height of the $N(n)$-th vertex in the corresponding GW-tree is bigger than $m(n)\sqrt{N(n)}/2$ has the desired bound.

At the current stage we need a relation between the height of a vertex of a forest and its L-walk. The L-walk of a forest is just the concatenation of all L-walks of the trees in the forest. A key fact is that the height of the $N$-th vertex has the same distribution as the number of `record times' of $L$ in $[[1,N]]$. More precisely, set $R_0=0$ and $R_j=\inf\{k>R_{j-1}:L(k)= \max_{i\in[[0,k]]}L(i)\}$. Then, the height of the $N$-th vertex in the forest of independent GW-tress has the same distribution as $\sharp\{j\in\N^+:R_j\leq N\}$. Moreover, the random variables $L(R_j)-L(R_{j-1}), j=1,2,\dots$ are i.i.d.\ with the distribution given by: $P(L(R_1)=i)=\mu_{-1}([i,\infty])$. Especially, $E[ L(R_1)]=\sigma^2/2$. For more details about this, see Section 1.3 in \cite{L05R}. With the help of these facts, we now get (write $k=\lceil m(n)\sqrt{N(n)}/2\rceil$)
\begin{align*}
P&[\sharp\{j\in\N^+:R_j\leq N\}\geq k]=P[R_k\leq N]\\
&\leq P[\sum_{j=1}^{k}(L(R_j)-L(R_{j-1}))\leq \frac{k\sigma^2}{4}]+ P[\sum_{j=1}^{k}(L(R_j)-L(R_{j-1}))\geq \frac{k\sigma^2}{4}, R_k\leq N]\\
&\leq P[\sum_{j=1}^{k}(L(R_j)-L(R_{j-1}))\leq \frac{k\sigma^2}{4}]+ P[\max_{i\leq N}\{L(i)\}\geq \frac{k\sigma^2}{4}].
\end{align*}
Now the last assertion in Lemma~\ref{lem-rw} gives the the desired bounds for both terms (note that $\mu_{-1}$ and $L(R_1)$ have finite exponential moments and that $E[L(R_j)-L(R_{j-1})]=\sigma^2/2$).
\end{proof}

\begin{lemma}\label{lm-sum}
For any $k,j\in\N^+$, let $\xi_1,\dots,\xi_j$ be i.i.d. with the distribution of $|\Tbf|$, and $\tilde{\xi}_i=\xi_i1_{\xi_i\leq k}$. There exists $C>0$, such that, when $t\geq Cj\sqrt{k}$,
$$
P[\tilde{\xi}_1+\dots+\tilde{\xi}_j>t]\leq \exp(-\frac{t}{k}).
$$
\end{lemma}
\begin{proof}
We know that $P(\xi_i=n)\asymp n^{-3/2}$. From this we get $E\tilde{\xi}_i\asymp \sqrt{k}$ and $E\tilde{\xi}_i^2\asymp k^{3/2}$.
A simple application of the Bernstein inequality gives the desired bound.
\end{proof}

\begin{prop}\label{cut-1}
For any $\alpha\in(0,1)$, there exist $C_1,C_2,c_1,c_2$ (depending on $\alpha$), such that, for any $n$ sufficiently large, $N(n)\in[[n^\alpha,n]]$, $b(n)\in[[C_1,\sqrt{n}]]$ and $a(n)\in [[C_2 (b(n))^2\log n,c_1(b(n))^2N(n)]]$, we have the following. For $I$, any interval in $[[0,n]]$ with length $N(n)$, let $\tbf_1,\dots,\tbf_j$ be all M-subtrees in $I$ of $\Tbf^n$. Then
$$
P[\sum_{|\tbf_i|\leq N(n)/a(n)}|\tbf_i|>\frac{N(n)}{b(n)}]\preceq \exp\left(-c_2\frac{a(n)}{(b(n))^{2}}\right).
$$
In other words, if we discard all small M-subtrees (with size at most $N(n)/a(n)$) in $I$, with high probability, we would not lose too many (bigger than $N(n)/b(n)$) vertices.
\end{prop}
\begin{proof}
As in the proof of Lemma~\ref{lem-numofst}, write $v_1$, $v_2$ for the first and the last vertices of $\Tbf$ in $I$ and $v'$ for the latest common ancestor of $v_1,v_2$. Assume $\tbf_1,\dots,\tbf_j$ are all $M$-subtrees in $I$. Let $l$ be the graph distance between $v_2$ and $v'$. Note that the set of all vertices in $\{V_i:i\in I\}$ except $v_1$ is the disjoint union of the set of vertices in $\tbf_1,\dots,\tbf_j$ and the set of vertices in the unique path connecting $v_2$ and $v'$ (including $v_2$ but not $v'$). Due to Lemma~\ref{lem-numofst}, by discarding an event with small probability, we can assume that
\begin{equation}\label{cs-eq3}
 j\leq c_0\frac{\sqrt{N}a^{1/2}}{b},\quad l\leq c_0\frac{\sqrt{N}a^{1/2}}{b},
\end{equation}
where we will pick up $c_0$ very small and then let $C_i$ be large enough and $c_i$ small enough (such that we can apply Lemma~\ref{lem-numofst}, Lemma~\ref{lm-sum} and Lemma~\ref{lem-rw}).

Note that we have
$$
\sum_{i=1}^j|\tbf_i|=N(n)-l.
$$
Moreover, it is elementary to see that conditionally on the value of $j$ and $l$, the sizes $(|\tbf_1|,\dots,|\tbf_j|)$ have the same distribution of $(\xi_1,\dots,\xi_j)$ in the last lemma conditioned on their sum being $N-l$. Hence we only need to estimate (under the same notations in the last lemma)
$$
P[\tilde{\xi}_1+\dots+\tilde{\xi}_j>t|\xi_1+\dots+\xi_j=N-l].
$$
with $k=N/a$ and $t=N/b$. By choosing $c_0$ small enough, $t,k,j$ can be satisfied with the requirement in Lemma~\ref{lm-sum}. Hence, we have
$$
P[\tilde{\xi}_1+\dots+\tilde{\xi}_j>\frac{N}{b}]\leq \exp(-\frac{a}{b}).
$$
On the other hand, we have (from \eqref{eq-Kem} and the third assertion of Lemma~\ref{lem-rw})
$$
P[\xi_1+\dots+\xi_j=N-l]\succeq \frac{j}{N^{3/2}}\exp(-c\frac{j^2}{N})\geq \frac{1}{N^{3/2}}\exp(-\frac{a}{2b}),
$$
by choosing $C_1$ and $C_2$ large enough.
Combining the two inequalities above, we get
$$
P[\tilde{\xi}_1+\dots+\tilde{\xi}_j>\frac{N}{b}|\xi_1+\dots+\xi_j=N-l]\preceq N^{3/2}\exp(-\frac{a}{2b}).
$$
By choosing $C_1$ large enough, we have
$$
N^{3/2}\exp(-\frac{a}{2b}) \leq \exp(-\frac{a}{b^2}).
$$
The proof is complete.
\end{proof}
\begin{remark}\label{rm-indep}
In this proposition, conditionally on the value of $(j;|\tbf_1|,\dots,|\tbf_j|)$, $\tbf_1,\dots,\tbf_j$ are actually independent and distributed as the GW-trees with their sizes. One can easily see this from the following. When conditioning on the values of the L-walk in $I$ outside the intervals corresponding to $\tbf_1,\dots,\tbf_j$, the L-walk restricted on the interval corresponding to $\tbf_i$ is just a conditioned random walk conditioned (only) on fixed starting point and end point and achieving its minimum only at the end point.
\end{remark}

\begin{prop}\label{cut-2}
For any positive $A>2,\epsilon,\alpha$ with $\alpha>2\epsilon$, there exists $C=C(\epsilon,\alpha, A)$, satisfying the following. When $n\in\N$ is sufficiently large, for any $k=k(n), L=L(n)\in[[1,n]]$ satisfying $k^A\geq n\geq L\log L, L\geq k^{2+\alpha}$, with probability bigger than $1-Ck^{1+\epsilon}/\sqrt{L}$, we can find a vertex $v$ in the $k$-th generation of $\Tbf^n$, such that
\begin{itemize}
    \item the subtree $[\Tbf^n]_v$ has at least $n-L$ vertices;
    \item conditionally on the size (even on the subtree generated by all edges not in $[\Tbf^n]_v$), $[\Tbf^n]_v$ is distributed as a GW-tree with the given size.
\end{itemize}
\end{prop}
\begin{proof}
For any tree $\tbf$, write $w_i(\tbf)$ for the number of vertices in the $i$-th generation and $\lceil \tbf\rceil_i$ for the subtree generated by all vertices of height $\leq i$. In order to prove this proposition, we need two lemmas.
\begin{lemma}\label{cs-lem1}
For any positive $\beta, \gamma$, there exists $C=C(\beta,\gamma)>0$, such that
\begin{align*}
    P[w_k(\Tbf^n)\geq k^{1+\beta}]\leq \frac{C}{k^\gamma},\\
    P[\sum_{i=1}^{k}w_i(\Tbf^n)\geq k^{2+\beta}]\leq \frac{C}{k^\gamma}.
\end{align*}
\end{lemma}
\begin{proof}
Lemma 2.2 in \cite{J06R} states that
$E(w_k(\Tbf^n))^r\leq C(r) k^{r-1}$. From this, one can get
\begin{equation}\label{cs-eq4}
P[w_k(\Tbf^n)\geq x]\leq \frac{C(r)k^{r-1}}{x^r}.
\end{equation}
Hence, we have
$$
P[w_k(\Tbf^n)\geq k^{1+\beta}]\leq\frac{C(r)k^{r-1}}{(k^{1+\beta})^r}=\frac{C(r)}{k^{1+\beta r}}.
$$
By choosing $r$ large enough, we get the first assertion.

For the second one,
\begin{align*}
P&[\sum_{i=0}^{k}w_i(\Tbf^n)\geq k^{2+\beta}]\leq
 \sum_{i=1}^{k}P[w_i(\Tbf^n)\geq k^{1+\beta}/2]\\
&\leq \sum_{i=1}^{k}C(r)\frac{i^{r-1}}{(k^{(1+\beta)}/2)^r}\leq kC(r)\frac{2^rk^{r-1}}{k^{(1+\beta)r}}=\frac{2^rC(r)}{k^{\beta r}}.
\end{align*}
By choosing $r$ large enough, we get the second assertion.
\end{proof}
As before, let $\xi_1,\xi_2,\dots,\xi_m$ be i.i.d.\ copies of $|\Tbf|$.
\begin{lemma}\label{cs-lem2}
For any $N\geq l\geq m\geq2$ with $N\geq 2l\geq m^2$,
\begin{equation}
    P[\max_{1\leq i\leq m}\{\xi_i\}\leq N-l|\sum_{i=1}^{m}\xi_i=N]\preceq \frac{m}{\sqrt{l}}+\frac{m}{\sqrt{N/\log N}}.
\end{equation}
\end{lemma}
\begin{proof}
We divide the term on the left hand side into three pieces and estimate them separately.
\begin{align*}
\sum_{j=l}^{\lfloor0.9N\rfloor}P[\max_{i}\{\xi_i\}=N-j&|\sum \xi_i=N]+\sum_{j=\lfloor N/(5\log N)\rfloor}^{\lfloor 0.1N\rfloor}P[\max_{i}\{\xi_i\}=j|\sum \xi_i=N]\\
&+P[\max_{i}\{\xi_i\}\leq N/(5\log N)|\sum \xi_i=N].
\end{align*}
The first one is bounded by
\begin{align*}
    \sum_{j=l}^{\lfloor 0.9N\rfloor}&\frac{P[\max\{\xi_i\}=N-j,\sum \xi_i=N]}{P[\sum\xi_i=N]}\preceq \sum_{j=l}^{\lfloor 0.9N\rfloor}\frac{mP[\xi_1=N-j,\sum_{i=2}^m\xi_i=j]}{\frac{m}{N}P[L(N)=-m]}\\
    &\preceq\sum_{j=l}^{\lfloor 0.9N\rfloor}\frac{mN^{-3/2}(m-1)j^{-3/2}}{mN^{-3/2}}\preceq\frac{m}{\sqrt{l}}.
\end{align*}
The second one can be analyzed similarly.
\begin{align*}
  \sum_{j=\lfloor N/(5\log N)\rfloor}^{\lfloor 0.1N\rfloor}&\frac{P[\max\{\xi_i\}=j,\sum \xi_i=N-j]}{P[\sum\xi_i=N]}\\
  &\preceq\sum_{j=\lfloor N/(5\log N)\rfloor}^{\lfloor 0.1N\rfloor}\frac{mP[\xi_1=j,\sum_{2}^m\xi_i=N-j]}{P[\sum\xi_i=N]}\\
  &\preceq\sum_{j=\lfloor N/(5\log N)\rfloor}^{\lfloor 0.1N\rfloor}\frac{mj^{-3/2}(m-1)N^{-3/2}}{mN^{-3/2}}\asymp\frac{m}{\sqrt{N/\log N}}.
\end{align*}
The last term is bounded by
$$
\frac{P[\max_{i}\{\xi_i\}\leq N/(5\log N),\sum \xi_i=N]}{P[\sum\xi_i=N]}\preceq \frac{P[\sum\xi_i1_{\xi_i\leq N/(5\log N)}\geq N ]}{mN^{-3/2}}.
$$
Applying Lemma~\ref{lm-sum} with $t=N,j=m,k=\lfloor N/(5\log N)\rfloor$, we obtain that the numerator is less than $N^{-5/2}$ and finish the proof.
\end{proof}
We are ready to prove our proposition. Let $\beta=\frac{1}{2}(\epsilon\wedge\alpha)$. By Lemma \ref{cs-lem1}, we can assume that (by discarding an event with small probability)
\begin{equation}\label{cs-eq6}
  M\doteq|\lceil \Tbf^n\rceil_k|\leq k^{2+\beta},\quad m\doteq w_k(\lceil \Tbf^n\rceil_k) \leq k^{1+\beta}.
\end{equation}
Write $v_1,\dots,v_m$ for all vertices in the $k$-generation and $\tbf_i$ for $[\Tbf^n]_{v_i}$ ($i=1,\dots,m$). As before, it is elementary to see that conditionally on $M,m$ (even on $\lceil \Tbf^n\rceil_k$), $|\tbf_1|,\dots,|\tbf_m|$ are distributed as $\xi_1,\dots,\xi_m$ conditioned on $\sum_{i=1}^{m}\xi_i=n-M+m$. Let $N=n-M+m$ and $l=L/2$. Applying Lemma \ref{cs-lem2}, we have (note that $M,m\ll L \ll n$, hence $n-M+m-L/2\geq n-L$)
$$
P[\max\{|\tbf_i|\}\geq n-L]\geq 1-C\frac{k^{1+\beta}}{\sqrt{L}}.
$$

Let $[\Tbf^n]_v$ be the one that achieves the maximum. The first requirement holds and the second is true due to the same reason as in Remark~\ref{rm-indep}. We finish the proof.
\end{proof}

We now state our result for decomposing a large random tree into small ones.
\begin{theorem}\label{cutting}
For any $d_0,e>0$, $\alpha,\beta,\delta\in(0,d_0)$ with $\epsilon_0\doteq\frac{3}{2}\beta-\alpha-d_0-\frac{7}{2}\delta>0$, $\epsilon\in(0,\epsilon_0)$, and $n(N)$, a function of $N$ with $n(N)\in[[N^{d_0}/(\log N)^e,N^{d_0}(\log N)^e]]$, there exists $C$ (depending on all parameters except $N$), such that, for any sufficiently large $N\in\N$, with
probability at least $1-CN^{-\epsilon}$, in $\Tbf^n$, we can find rooted subtrees $T_1,\dots T_{m}$ (note that $m$ is also random), where $T_i$ is rooted at $v_i$, the unique vertex in $T_i$ closest to the root of $\Tbf^n$, $o$, satisfying the following:
\begin{enumerate}
   \item For every $i\in\{1,...,m\}$, $N^{\beta-2\delta}/(\log N)^{1.01}\leq |T_i| \leq N^{\beta}$ and the distance between $v_i$ and $o$ is at least $\lfloor N^{\alpha}\rfloor$;
   \item Let $\hat{T}$ be the graph generated by all edges not in any $T_i$. Then $\hat{T}$ is a tree and $|\hat{T}|\leq n/N^{\delta}$;
   \item Let $\iota_i$ ($i=1,\dots,m$) be the unique path starting from $v_i$ towards the root of $T$, with length $\lfloor N^{\alpha}\rfloor$. Then for any $i\in\{1,\dots,m\}$, all $T_j$ except $T_i$, are in the same component of $T\setminus \iota_i$;
   \item Conditionally on $\{m;|T_1|,\dots,|T_{m}|;\hat T\}$ (and even the locations of $v_i$ in $\hat T$), the trees $T_i$ are independent and distributed as the GW-trees with their sizes.
\end{enumerate}
\end{theorem}
\begin{proof}
This theorem is a combination of Proposition~\ref{cut-1} and Proposition~\ref{cut-2}. We only need to designate the parameters.

Note that since $\epsilon_0>0$, we have $0<d_0-\beta+2\delta<\beta/2-\alpha-3\delta/2$ and $(\beta/2-\alpha-3\delta/2)-(d_0-\beta+2\delta)=\epsilon_0$. Let $\gamma=d_0-\beta+2\delta+(\epsilon+\epsilon_0)/2$. Note that $\beta/2-\alpha-3\delta/2-\gamma=(\epsilon_0-\epsilon)/2$. Set $a=N^{2\delta}(\log N)^{1.01}/3,b=2N^\delta$, $k=\lfloor N^\alpha\rfloor$ and $L=N^{2\alpha+2\gamma}/2$. First, divide $[[0,n]]$ (deterministically) into disjoint intervals with lengths in $[[N^\beta/2,N^\beta]]$. For each interval, apply Proposition~\ref{cut-1}. Then, (with high enough probability, since the error term decays faster than any polynomial of $N$) without losing too many vertices ($<n/2N^\delta$), we obtain a number of M-subtrees, say $T'_1,\dots,T'_{m}$, with sizes in $[[3N^{\beta-2\delta}/2(\log N)^{1.01},N^\beta]]$. Then, for each $T'_i$, by applying Proposition~\ref{cut-2}, we can obtain, with high probability($\geq 1-Ck^{1+0.1(\epsilon_0-\epsilon)/d_0}/\sqrt{L}\geq 1-CN^{0.1(\epsilon_0-\epsilon)}/N^\gamma$), a subtree $T_i$, which is at distance $k$ from the root of $T'_i$, without losing more than $L$ vertices. We argue that $T_1,\dots,T_m$ satisfy all conditions. Conditions (3), (4), and part of (1) and (2) are satisfied automatically from the propositions. We only need to check the quantitative parts of Conditions (1) and (2).

First, since $2\alpha+2\gamma=(\beta-2\delta)-\delta-(\epsilon_0-\epsilon)$, one can see that the size of $T_i$ is at least
$$
\frac{3N^{\beta-2\delta}}{2(\log N)^{1.01}}-L\geq \frac{3N^{\beta-2\delta}}{2(\log N)^{1.01}}-\frac{N^{\beta-2\delta}}{2(\log N)^{1.01}}=\frac{N^{\beta-2\delta}}{(\log N)^{1.01}}.
$$
Hence, Condition (1) is satisfied.

Second, consider the total number of vertices that we lose. In the first step (when applying Proposition~\ref{cut-1}), we lose at most $n/2N^{\delta}$ vertices. On the other hand, note that (when $N$ is large enough)
$$
m<\frac{n}{3N^{\beta-2\delta}/2(\log N)^{2.001}}\leq\frac{n}{N^{\beta-2\delta-0.1(\epsilon_0-\epsilon)}}.
$$
Therefore, in the second step (when applying Proposition~\ref{cut-2}), the number of vertices we lose is at most
\begin{multline*}
mL\leq \frac{n}{2N^{\beta-2\delta-2\alpha-2\gamma-0.1(\epsilon_0-\epsilon)}}\\
=\frac{n}{2N^{\delta+2(\beta/2-\alpha-3\delta/2-\gamma)-0.1(\epsilon_0-\epsilon)}}=\frac{n}{2N^{\delta+0.9(\epsilon_0-\epsilon)}}<\frac{n}{2N^\delta}.
\end{multline*}
Combining both, one can get Condition (2).

Finally, note that
$$
m\frac{N^{0.1(\epsilon_0-\epsilon)}}{N^\gamma}\leq \frac{N^{d_0}}{N^{\gamma+(\beta-2\delta)-0.3(\epsilon_0-\epsilon)}}=\frac{1}{N^{0.2\epsilon_0+0.8\epsilon}}\leq \frac{1}{N^\epsilon}.
$$
It means that with probability at least $1-CN^{-\epsilon}$, we can get all $T_i$ from $T_i'$ as desired. The proof is complete.
\end{proof}

\section{The local limit of tree-indexed random walks in tori}
In this and the next sections, we consider tree-indexed random walks in the discrete torus $\Torus_N\doteq(\Z/N\Z)^d$ of side-length $N$ ($d\geq 5$). For technical reasons, we assume that the jump distribution $\theta$ is aperiodic, symmetric, with finite range and not supported on any strict subgroup of $\Z^d$, and that the offspring distribution $\mu$ is critical, nondegenerate, with finite exponential moments. For simplicity, we assume further that $\mu$ is with span $1$. The case when the span is bigger than $1$ can be treated with minor standard modifications. We aim to show the local convergence result in this section. Write $\varphi:\Z^d\rightarrow \Torus_N$ for the canonical projection map induced by $\mathrm{mod~}N$. We now state our main theorem.
\begin{theorem}\label{ll-thm-main}
For any $K\ssubset \Z^d$, $u\in(0,\infty)$ and $n=n(N)$, an integer-valued function satisfying $\lim_{N\rightarrow \infty}n(N)/N^d=u$,
\begin{equation}\label{eq-main}
\lim_{N\rightarrow \infty}P[\mathrm{Range}(\Snake^{n,N})\cap \varphi(K)=\emptyset]=\exp(-u\BCap(K)),
\end{equation}
where $\Snake^{n,N}$ is the tree-indexed random walk with size $n$ and a uniform starting point in $\Torus_N$.
\end{theorem}
Through the inclusion-exclusion principle, this theorem implies the local convergence of the configuration, Theorem~\ref{thm-ll}
\subsection{The visiting probability of a set by a small tree-indexed random walk}
Theorem~\ref{ll-thm-main} gives an asymptotic formula for the probability that a tree-indexed random walk visits a subset in $\Torus_N$, with a size proportional to the volume of the torus, $N^d$. The main result of this subsection is to give an asymptotic formula for the probability that a set is visited by a much smaller tree-indexed random walk.
\begin{theorem}\label{visit-small}
For any $\alpha_1<\alpha_2\in(0,d)$ fixed, $n=n(N)$ any integer-valued function of $N$ satisfying $n(N)\in[N^{\alpha_1},N^{\alpha_2}]$, and $K\ssubset \Z^d$, we have
\begin{equation}\label{eq-visit-small}
    \lim_{N\rightarrow\infty}\frac{N^d}{n}P[\mathrm{Range}(\Snake^{n,N})\cap \varphi(K)\neq\emptyset]=\BCap(K).
\end{equation}
\end{theorem}


We start with our notations. For any $m\in \N^+$ and $\tbf\in \T_\mathrm{f}\cup\T_\infty$ with $|\tbf|>m$, denote by $\rho_m(\tbf)$ and $\rho^+_m(\tbf)$, respectively, the subtree generated by the first $m$ vertices of $\tbf$ (i.e.\ $V_0,\dots,V_{m-1}$) and the one generated by the first $m$ vertices together with all children of the first $m$ vertices. Note that $\rho^+_m(\tbf)$ is uniquely determined by $\rho_m(\tbf)$ and the degrees of the vertices that are ancestors of $V_{m-1}$. In addition, since the parent of a vertex comes before that vertex, we can see that $\rho^+_m(\tbf)$ is uniquely determined by $(L_\tbf(i))_{i\in[[0,m]]}$. Note that in $\rho^+_m(\tbf)$, there are exactly $L_\tbf(m)+1$ vertices, which are not in $\rho_m(\tbf)$. We denote by $\partial \rho_m(\tbf)$ the set of these vertices.

As before, write $\Tbf^n$ and $\Tbf^\infty$ for the GW-tree with $n$ vertices and the one conditioned on survival respectively. We now compare $P[\rho^+_m(\Tbf^n)=\rho^+_m(\tbf)]$ to $P[\rho^+_m(\Tbf^\infty)=\rho^+_m(\tbf)]$. By considering the L-walks, it is elementary to see that
$$
P[\rho^+_m(\Tbf^n)=\rho^+_m(\tbf)]=\frac{P[L(i)=L_\tbf(i),i=0,1,\dots,m; H_{-1}=n]}{P[|\Tbf|=n]},
$$
where as before, $L=(L(i))_{i\in\N}$ is a random walk starting at $0$ with jump distribution $\mu_{-1}$, $H_{-1}$ is the fist visiting time of $-1$ and $\Tbf$ is an unconditioned GW-tree. Using the Markov property at time $m$ and \eqref{eq-Kem}, one can get
\begin{multline}\label{eq-rho1}
 P[\rho^+_m(\Tbf^n)=\rho^+_m(\tbf)]=\\
 \frac{P[L(i)=L_\tbf(i),i=0,1,\dots,m]\frac{L_\tbf(m)+1}{n-m}P[L(n-m)=-(L_\tbf(m)+1)]}{P[|\Tbf|=n]}.
\end{multline}

We turn to $\Tbf^\infty$. Note that when $\rho^+_m(\Tbf^\infty)=\rho^+_m(\tbf)$, there is a unique vertex in $\partial\rho_m(\tbf)$, which is also in the spine of $\Tbf^\infty$. On the other hand, for any $u\in\partial\rho_m(\tbf)$, we have
\begin{align*}
 P[\rho^+_m(\Tbf^\infty)=\rho^+_m(\tbf),\;u &\mathrm{~is~a~vertex} \mathrm{~in~the~spine~of~}\Tbf^\infty]\\
& =P[L(i)=L_\tbf(i),i=0,1,\dots,m].
\end{align*}
Hence, we get
\begin{equation}\label{eq-rho2}
P[\rho^+_m(\Tbf^\infty)=\rho^+_m(\tbf)]=P[L(i)=L_\tbf(i),i=0,1,\dots,m](L_\tbf(m)+1).
\end{equation}
Combining \eqref{eq-rho1} and \eqref{eq-rho2} leads to
$$
\frac{P[\rho^+_m(\Tbf^n)=\rho^+_m(\tbf)]}{P[\rho^+_m(\Tbf^\infty)=\rho^+_m(\tbf)]}=\frac{P[L(n-m)=-(L_\tbf(m)+1)]}{(n-m)P[|\Tbf|=n]}.
$$
When $n\rightarrow \infty$ and $m/n\rightarrow a\in(0,1)$, by the first assertion of Lemma \ref{lem-rw}, we get
$$
\frac{P[\rho^+_m(\Tbf^n)=\rho^+_m(\tbf)]}{P[\rho^+_m(\Tbf^\infty)=\rho^+_m(\tbf)]}=\Gamma_a(\frac{L_\tbf(m)+1}{\sqrt{n}})+o(1),
$$
where $\Gamma_a(x)=\exp(-x^2/2\sigma^2(1-a))/(1-a)^{3/2}$. Note that $\Gamma_a$ is bounded above, when $a\in(0,1)$ is fixed. Hence, we have, when $m=\lfloor an\rfloor$,
\begin{equation}\label{ll-eq-n-inf1}
  P[\rho^+_m(\Tbf^n)=\rho^+_m(\tbf)]\leq C(a)P[\rho^+_m(\Tbf^\infty)=\rho^+_m(\tbf)].
\end{equation}

We now give a similar relation between $\Tbf^\infty$ and $\Tbf^\mathrm{c}$, the random tree with law $\tc$. Note that $\Tbf^\infty$ and $\Tbf^c$ differ only in the finite bushes attached to the roots. Write $v$ for the vertex in the spine next to the root. Then, for any $i\in\N$,
\begin{align*}
   P&[v \mathrm{~has~exact~} i \mathrm{~elder~siblings~in~ }\Tbf^\infty]=\sum_{j=i+1}^\infty\mu(j)\\
   &\leq \sum_{j=i}^\infty\mu(j)=2 P[v \mathrm{~has~exact~} i \mathrm{~elder~siblings~in~}\Tbf^\mathrm{c}].
\end{align*}
Hence, we have
$$
P[\rho_m(\Tbf^\infty)=\rho_m(\tbf)]\leq 2P[\rho_m(\Tbf^\mathrm{c})=\rho_m(\tbf)].
$$
Combining this with \eqref{ll-eq-n-inf1}, we get
\begin{lemma}\label{ll-lm-n-c}
\begin{equation}\label{ll-eq-n-inf}
    P[\rho^+_m(\Tbf^n)=\rho^+_m(\tbf)]\leq C(a)P[\rho^+_m(\Tbf^\infty)=\rho^+_m(\tbf)],
\end{equation}
\begin{equation}\label{ll-eq-n-c}
    P[\rho_m(\Tbf^\infty)=\rho_m(\tbf)]\preceq P[\rho_m(\Tbf^\mathrm{c})=\rho_m(\tbf)].
\end{equation}
\end{lemma}

Before the formal proof, let us set up our notations. For any (finite or infinite) spatial tree $(\tbf,\Snake)$, write $X_i=X_i(\tbf,\Snake)=\Snake(V_i(\tbf))$. We let $P^n_{k,x}$ be the law when $(\tbf,\Snake)$ is as follows. First $\tbf$ is a GW-tree with size $n$ and conditionally on $\tbf$, $\Snake$ is a $\tbf$-indexed random walk sending $V_k$ to $x$, i.e.\ the law of $\Snake:\tbf\rightarrow \Z^d$ is determined by:
$V_k$ is mapped to $x$ and the random variables $(\Snake(v_2)-\Snake(v_1))_{(v_1,v_2)}$ are independent and distributed according to $\theta$, where $(v_1,v_2)$ are over all (undirected) edges of $\tbf$.
Similarly, one can define $P^\infty_{k,x}$ and $P^\mathbf{v}_{k,x}$. We write $P^{n,N}$ for the law when $(\tbf,\Snake)$ is the tree-indexed random walk with size $n$, in $\Torus_N$, with a uniform starting point.

\begin{proof}[Proof of Theorem \ref{visit-small}]
Write $K'=\varphi(K)$. We have (when $N$ is large enough)
\begin{align*}
    P^{n,N}&[\{X_0,\dots,X_{n-1}\}\cap K'\neq\emptyset]\times \frac{N^d}{n}\\
    =&\sum_{x\in  K'}\sum_{k=1}^{n-1}P^{n,N}[X_0,\dots,X_{k-1}\notin K';X_k=x]\frac{N^d}{n}+\sum_{x\in  K'}P^{n,N}[X_0=x]\frac{N^d}{n}\\
    =&\sum_{x\in  K'}\left(\frac{1}{n}\sum_{k=1}^{n-1}N^dP^{n,N}[X_0,\dots,X_{k-1}\notin K';X_k=x]\right)+\frac{|K'|}{n}.
\end{align*}
It suffices to show that for any $x\in K$ and $\epsilon\in(0,0.1)$,
$$
\lim_{N\rightarrow \infty}\frac{1}{n}\sum_{\epsilon n<k<(1-\epsilon) n}|N^dP^{n,N}[X_0,\dots,X_{k-1}\notin K';X_k=\varphi(x)]-\Es_K(x)|=0.
$$
The equality above follows if we can show the following three lemmas.
\begin{lemma}\label{ll-l1}
\begin{multline}\label{eq-l1}
  \lim_{N\rightarrow \infty}\frac{1}{n}\sum_{\epsilon n<k<(1-\epsilon) n}|N^dP^{n,N}[X_0,\dots,X_{k-1}\notin K';X_k=\varphi(x)]-\\
  P_{k,x}^{n}[X_0,X_1,\dots,X_{k-1}\notin K]|=0.
\end{multline}
\end{lemma}
\begin{lemma}\label{ll-l2}
For any $l=l(n)\in[[1,\epsilon n]]$ with $l(n)\rightarrow \infty$,
\begin{multline}\label{eq-l2}
    \lim_{n\rightarrow \infty}\frac{1}{n}\sum_{\epsilon n<k<(1-\epsilon) n}|P_{k,x}^{n}[X_0,X_1,\dots,X_{k-1}\notin K]-\\
    P_{k,x}^{n}[X_{k-l},X_{k-l+1},\dots,X_{k-1}\notin K]|=0.
\end{multline}
\end{lemma}
\begin{lemma}\label{ll-l3}
For any $l=l(n)\in[[1,\epsilon n]]$ with $1\ll l(n)\ll \sqrt{n}$,
\begin{equation}\label{eq-l3}
      \lim_{n\rightarrow \infty}\sup_{\epsilon n<k<(1-\epsilon) n}|P^n_{k,x}[X_{k-l},X_{k-l+1},\dots,X_{k-1}\notin K]-\Es_K(x)|=0.
\end{equation}
\end{lemma}
We need the following two lemmas in the proof of Lemma~\ref{ll-l1}.
\begin{lemma}\label{large-range}
There exist $C,c_1,c_2$, such that for any $n,l\in\N^+$ with $l\in[C_1n^{1/4}(\log n)^{3/4},c_1\sqrt{n}]$,
\begin{equation}
\sup_{0\leq i\leq n}P_{0,x}^{n}[|X_i-x|>l]\preceq \exp(-c_2\frac{l^{4/3}}{n^{1/3}}).
\end{equation}
\end{lemma}
\begin{lemma}\label{snake-x-y}
For any $x\neq y\in\Z^d$ and $n\in \N^+$,
\begin{equation}\label{eq-snake-x-y}
\frac{1}{n}\sum_{0\leq i<j<n}P^n_{j,x}(X_i=y)\preceq |x-y|^{4-d}
\end{equation}
\end{lemma}

\begin{proof}[Proof of Lemma~\ref{large-range}]
Let $h$ be the height of $V_i$ in $\Tbf^n$. By Lemma~\ref{lem-numofst}, we have
$$
P[h\geq l^{2/3}n^{1/3}]\preceq \exp(-c\frac{l^{4/3}}{n^{1/3}}).
$$
On the other hand, by standard estimates for random walks (e.g.\ the last assertion in Lemma~\ref{lem-rw}) we have
$$
P^n_{0,x}[|X_i-x|>l|h\leq l^{2/3}n^{1/3}]\preceq \exp(-c\frac{l^{4/3}}{n^{1/3}}).
$$
Combining both displays finishes the proof.
\end{proof}
\begin{proof}[Proof of Lemma~\ref{snake-x-y}]
By conditioning on the distance between $V_i$ and $V_j$, one can get
\begin{align*}
    \frac{1}{n}\sum_{0\leq i<j<n}P^n_{j,x}(X_i=y)&=\frac{1}{n}\sum_{0\leq i<j<n}\sum_{k\in\N^+}P^n_{j,x}[d(V_i,V_j)=k]P^n_{j,x}[X_i=y|d(V_i,V_j)=k]\\
    &=\frac{1}{n}\sum_{k\in\N^+}P^\mathrm{RW}_x(Z(k)=y)\sum_{0\leq i<j<n}P^n_{j,x}[d(V_i,V_j)=k],\\
\end{align*}
where $P^\mathrm{RW}_x$ indicates that $Z=(Z(i))$ is a random walk starting at $x$ (with jump distribution $\theta$).

At the current stage, we need the following result, Theorem 1.1 in \cite{DJ11D}, which will also be useful later.
\begin{lemma}\label{J2}
There exists a constant $C$, such that for all $n,k\in\N^+$, the expected number of pairs of vertices with distance $k$, in $\Tbf^n$, is at most $Ckn$.
\end{lemma}

Hence
$$
\frac{1}{n}\sum_{0\leq i<j<n}P^n_{j,x}[d(V_i,V_j)=k]\preceq k.
$$
Therefore,
$$
\frac{1}{n}\sum_{0\leq i<j<n}P^n_{j,x}(X_i=y)\preceq \sum_{k\in\N^+}kP^\mathrm{RW}_x(Z(k)=y)\asymp |x-y|^{4-d},
$$
where the last step is standard.
\end{proof}

\begin{proof}[Proof of Lemma \ref{ll-l1}]
Let $\varphi^{-1}(K')=\bigcup_{i=0}^{\infty}K_i$, such that $x\in K_0=K$ and
$K_i$ is a translated copy of $K_0$.
Fix some $\lambda \in (\frac{1}{4},\frac{d}{4\alpha_2})$ and let $b=\lfloor \frac{n^\lambda }{N}\rfloor +1$.
\begin{align*}
&\frac{1}{n}\sum_{\epsilon n<k<(1-\epsilon) n}|N^dP^{n,N}[X_0,\dots,X_{k-1} \notin K;X_k=\varphi(x)]-P_{k,x}^{n}[X_0,\dots,X_{k-1}\notin K]|\\
=&\frac{1}{n}\sum_{\epsilon n<k<(1-\epsilon) n}P_{k,x}^{n}[X_0,X_1,\dots,X_{k-1}\notin K]-P_{k,x}^{n}[X_0,X_1,\dots,X_{k-1} \notin \varphi^{-1}(K)]\\
\leq &\frac{1}{n}\sum_{\epsilon n<k<(1-\epsilon) n}P_{k,x}^{n}[\{X_0,\dots,X_{k-1}\}\cap (\cup_{i\geq1}K_i)\neq \emptyset]\\
\leq &\frac{1}{n}\sum_{\epsilon n<k<(1-\epsilon) n}P_{k,x}^{n}[\sup_{0\leq i\leq n-1}|X_i-x|>bN]+\\
&\quad\quad \frac{1}{n}\sum_{\epsilon n<k<(1-\epsilon) n}P_{k,x}^{n}[\sup_{0\leq i\leq n-1}|X_i-x|\leq bN, \{X_0,\dots,X_{k-1}\}\cap (\cup_{i\geq1}K_i)\neq \emptyset].
\end{align*}
The first term above goes to $0$, due to Lemma~\ref{large-range} (since $bN\geq n^{\lambda},\lambda >1/4$).

For the other term, we have when $N$ is large, (write $B_x(r)$ and $S_x(r)$ for the box of radius $r$ centered at $x$ and the boundary of $B_x(r)$ respectively)
\begin{align*}
&\frac{1}{n}\sum_{\epsilon n<k<(1-\epsilon) n}P_{k,x}^{n}[\sup_{0\leq i\leq n-1}|X_i-x|\leq bN, \{X_0,\dots,X_{k-1}\}\cap (\cup_{i\geq1}K_i)\neq \emptyset]\\
\leq &\frac{1}{n}\sum_{\epsilon n<k<(1-\epsilon) n}\sum_{i:K_0\neq K_i\subseteq B_x\left((b+1)N\right)}\sum_{y\in K_i}P_{k,x}^{n}[y\in\{X_0,...X_{k-1}\}]\\
\leq &\sum_{i=1}^{b+1}\sum_{j:K_j\cap S_x(iN)\neq \emptyset}\frac{1}{n}\sum_{\epsilon n<k<(1-\epsilon) n}\sum_{y\in K_j}\sum_{0\leq l<k}P_{k,x}^{n}[X_l=y]\\
\stackrel{\eqref{eq-snake-x-y}}{\preceq} &\sum_{i=1}^{b+1}\sum_{j:K_j\cap S_x(iN)\neq \emptyset}\sum_{y\in K_j}\frac{1}{(iN)^{d-4}}\preceq \sum_{i=1}^{b+1}i^{d-1}\cdot \frac{|K|}{(iN)^{d-4}}\\
\preceq &|K|\frac{b^4}{N^{d-4}}\preceq |K|\frac{n^{4\lambda}/N^4+1}{N^{d-4}}\rightarrow 0,
\end{align*}
where the last convergence follows from $\lambda< \frac{d}{4\alpha_2}$, $\alpha_2<d$ and $n \leq N^{\alpha_2}$.
\end{proof}

\begin{proof}[Proof of Lemma~\ref{ll-l2}]
Note that
\begin{align*}
0\leq& P_{k,x}^{n}[X_{k-l},X_{k-l+1},\dots,X_{k-1}\notin K]-P_{k,x}^{n}[X_0,\dots,X_{k-1}\notin K]\\
    \leq& P_{k,x}^{n}[\{X_0,\dots,X_{k-l}\}\cap K\neq\emptyset]\leq\sum_{y\in K}P_{k,x}^{n}[y\in \{X_0,\dots,X_{k-l}\}]\\
    \leq &C(\epsilon) \sum_{y\in K}P_{k,x}^{\infty}[y\in \{X_0,\dots,X_{k-l}\}],
\end{align*}
where for the last step, we use \eqref{ll-eq-n-inf}.

It would be ideal if we can establish a formula similar to \eqref{ll-eq-n-c}, for $\Tbf^\infty$ and $\Tbf^\mathrm{v}$, the random tree with law $\tv$. But in this case, the following inequality is not generally true.
$$
P[v \mathrm{~has~exact~} i \mathrm{~elder~siblings~in~}\Tbf^\infty]\preceq P[v \mathrm{~has~exact~} i \mathrm{~elder~siblings~in~}\Tbf^\mathrm{v}].
$$
The following inequality is enough for our purpose.
\begin{lemma}
\begin{equation}\label{ll-eq-inf-v}
P_{k,x}^{\infty}[A]\preceq P_{k,x}^{\mathrm{v}}[A]+P_{k,x}^{\infty}[X_0=y],
\end{equation}
where $A$ is the event $y\in \{X_0,\dots,X_{k-l}\}$.
\end{lemma}
\begin{proof}
Write $A_0$ for the event $y\in\{X_1,\dots,X_{k-l}\}$. We have $P_{k,x}^{\infty}[A]\leq P_{k,x}^{\infty}[A_0]+P_{k,x}^{\infty}[X_0=y]$. For any $\tbf\in\Tinf$ and $m\in\N^+$, write $\rho'_m(\tbf)$ for the marked tree $\rho_m(\tbf)$ with those vertices that are in the spine of $\tbf$, marked. Write $M$ for the set of all marked trees $\tbf$ such that $P[\rho'_{k+1}(\Tbf^\infty)=\tbf]>0$. Note that
\begin{eqnarray*}
P_{k,x}^{\infty}[A]=\sum_{\tbf\in M}P[\rho'_{k+1}(\Tbf^\infty)=\tbf]P_{k,x}^{\infty}[A|\rho'_{k+1}(\Tbf^\infty)=\tbf],\\
P_{k,x}^\mathrm{v}[A]\geq\sum_{\tbf\in M}P[\rho'_{k+1}(\Tbf^\mathrm{v})=\tbf]P_{k,x}^\mathrm{v}[A|\rho'_{k+1}(\Tbf^\mathrm{v})=\tbf].
\end{eqnarray*}
By a slight abuse of notation, we also write $\Tbf^\infty$ ($\Tbf^\mathrm{v}$) for the tree in the tree-indexed random walk corresponding to $P_{k,x}^{\infty}$($P_{k,x}^\mathrm{v}$).

We divide $M$ into three disjoint subsets:
\begin{eqnarray*}
&&M_1=\{\tbf\in M: o\mathrm{~has~exactly~one~child~and ~this~child~is~marked}\},\\
&&M_2=\{\tbf\in M: o\mathrm{~has~exactly~one~child~and ~this~child~is~not~marked}\},\\
&&M_3=\{\tbf\in M: o\mathrm{~has~}\geq 2\mathrm{~children}\}.
\end{eqnarray*}
Obviously we have
\begin{equation*}
P_{k,x}^{\infty}[A|\rho'_{k+1}(\Tbf^\infty)=\tbf]=P_{k,x}^{\mathrm{v}}[A|\rho'_{k+1}(\Tbf^\mathrm{v})=\tbf], \mathrm{~for~}\tbf\in M_1\cup M_2.
\end{equation*}
Moreover, it is easy to see
\begin{eqnarray*}
P[\rho'_{k+1}(\Tbf^\infty)=\tbf]=\frac{1-\mu(0)}{\mu(0)}P[\rho'_{k+1}(\Tbf^\mathrm{v})=\tbf], \mathrm{~for~}\tbf\in M_1,\\
P[\rho'_{k+1}(\Tbf^\infty)=\tbf]=\frac{\sum_{i\geq 2}(i-1)\mu(i)}{\sum_{i\geq 1}\mu(i)}P[\rho'_{k+1}(\Tbf^\mathrm{v})=\tbf], \mathrm{~for~}\tbf\in M_2.
\end{eqnarray*}
Hence we obtain
\begin{align*}
\sum_{\tbf\in M_1\cup M_2}&P[\rho'_{k+1}(\Tbf^\infty)=\tbf]P_{k,x}^{\infty}[A|\rho'_{k+1}(\Tbf^\infty)=\tbf]\preceq\\
&\sum_{\tbf\in M_1\cup M_2}P[\rho'_{k+1}(\Tbf^\mathrm{v})=\tbf]P_{k,x}^{\mathrm{v}}[A|\rho'_{k+1}(\Tbf^\mathrm{v})=\tbf]\leq P_{k,x}^\mathrm{v}[A].
\end{align*}
It suffices to show that
$$
\sum_{\tbf\in M_3}P[\rho'_{k+1}(\Tbf^\infty)=\tbf]P_{k,x}^{\infty}[A_0|\rho'_{k+1}(\Tbf^\infty)=\tbf]\preceq P_{k,x}^{\mathrm{v}}[A].
$$

For any $\tbf\in M_3$, write $f(\tbf)$ for $\tbf$ re-rooted at the first child of $o$ with the new root added to the set of marked vertices (and $o$ being regarded as the last child of the new root). It is easy to check that
\begin{equation*}
   P_{k,x}^{\infty}[A_0|\rho'_{k+1}(\Tbf^\infty)=\tbf]\leq P_{k,x}^{\mathrm{v}}[A|\rho'_{k+1}(\Tbf^\mathrm{v})=f(\tbf)], \mathrm{~for~}\tbf\in M_3.
\end{equation*}
On the other hand, one can verify that
$$
P[\rho'_{k+1}(\Tbf^\infty)=\tbf]\leq P[\rho'_{k+1}(\Tbf^\mathrm{v})=f(\tbf)],\mathrm{~for~}\tbf\in M_3.
$$
Hence, we have
\begin{align*}
 \sum_{\tbf\in M_3}&P[\rho'_{k+1}(\Tbf^\infty)=\tbf]P_{k,x}^{\infty}[A_0|\rho'_{k+1}(\Tbf^\infty)=\tbf]\\
 &\leq \sum_{\tbf\in M_3} P[\rho'_{k+1}(\Tbf^\mathrm{v})=f(\tbf)] P_{k,x}^{\mathrm{v}}[A|\rho'_{k+1}(\Tbf^\mathrm{v})=f(\tbf)]\leq P_{k,x}^\mathrm{v}[A],
\end{align*}
Note that we use the fact that $f$ is injective for the last step. The proof of \eqref{ll-eq-inf-v} is complete.
\end{proof}

Thanks to the last lemma, (note that $P_{k,x}^{\infty}[X_0=y]=P_{0,y}^{\infty}[X_k=x]\stackrel{k\rightarrow \infty}{\longrightarrow}0$,) it suffices to show, for any $y\in K$,
\begin{equation}\label{cc}
\lim_{n\rightarrow \infty}\frac{1}{n}\sum_{\epsilon n<k<(1-\epsilon) n}P_{k,x}^\mathrm{v}[y\in\{X_0,\dots,X_{k-l}\}]=0.
\end{equation}

Write $B=B(k)$ and $\Snake(B)$ for the set of non-root vertices in the spine, that are not after $V_k$, and the range of $B$, respectively. It is standard that $P_{k,x}^\mathrm{v}[V_k\in B]=o(1)$ (when $k\rightarrow \infty$). Hence we have
\begin{equation}\label{cc1}
  P_{k,x}^\mathrm{v}[y\in\{X_0,\dots,X_{k-l}\}]\leq P_{k,x}^\mathrm{v}[y\in\{X_0,\dots,X_{k-l}\}, V_k\notin B]+o(1).
\end{equation}
By changing the spatial tree a bit, around the fist visiting point of $y$ by $\Snake(B)$, one can show that for any $i_0\in\N^+$ fixed, with $\mu(i_0)>0$,
\begin{multline*}
P_{k,x}^\mathrm{v}[y\in\Snake(B)\setminus (\Snake(\{V_0,\dots,V_{k-1}\}\setminus B)),V_k\notin B]\\
\preceq \sum_{i=0,i_0} P_{k+i,x}^\mathrm{v}[y\in\Snake(\{V_0,\dots,V_{k-1}\}\setminus B),V_k\notin B].
\end{multline*}
Here is a naive argument. Assume that $(\tbf, \Snake)$ is a spatial tree in the event on the left hand side. We need to argue that we can make $(\tbf, \Snake)$ visit $y$ for some vertex before $V_k$, not in the spine, by changing $(\tbf, \Snake)$ a bit, e.g.\ by changing one variable in some edge of $\tbf$ or inserting $i_0$ new edges and the corresponding variables in the new edges. Let $v\in B$ be the first vertex such that $\Snake(v)=y$ and $v_1$ be the parent of $v$. If $v$ has elder siblings, then just pick up the eldest one and change the variable in the edge connecting $v_1$ and that vertex such that the sum of those two variables in the edges connecting $v$ and that vertex is zero (we can do so since $\theta$ is symmetric). If $v$ has no elder siblings, we can achieve our goal by inserting $i_0$ elder siblings for $v$ and sending the eldest one to $y$.

From the last inequality, one can obtain
\begin{multline}\label{cc2}
P_{k,x}^\mathrm{v}[y\in \{X_0,\dots,X_{k-l}\}, V_k\notin B]\\
\preceq \sum_{i=0,i_0} P_{k+i,x}^\mathrm{v}[y\in \Snake(\{V_0,\dots,V_{k-1}\}\setminus B), V_k\notin B].
\end{multline}
Write $\Bbs(a,b)=\Bbs(a,b;\tbf)$ for the number of spine vertices in $\{V_i(\tbf):i\in[[a,b]]\}$. The following estimate of $\Bbs(a,b;\Tbf^v)$ is standard and easy:
\begin{equation}\label{cc3}
 \lim_{n\rightarrow\infty}\sup_{m\in\N^+}P[\Bbs(m,m+n;\Tbf^v)\geq n/2]=0.
\end{equation}

Joining \eqref{cc} with \eqref{cc1}, \eqref{cc2} and \eqref{cc3}, one can see that it suffices to show that (by making $\epsilon$ smaller)
\begin{multline}\label{rd}
\lim_{n\rightarrow \infty}\frac{1}{n}\sum_{\epsilon n<k<(1-\epsilon) n}P_{k,x}^\mathrm{v}[y\in\Snake(\{V_0,\dots,V_{k-1}\}\setminus B(k)), \\
V_k\notin B(k), \Bbs(k-l,k)\leq l/2]=0.
\end{multline}

Let us introduce our notations. Recall that for $\tbf\in\Tinf$, we have defined the depth-first search from infinity. Write $(V^\infty_i(\tbf))_{i\in\Z}$ for all vertices arranged with that order such that $V_0^\infty(\tbf)$ is the root (we will drop $\tbf$ in the notation when $\tbf$ is obvious, as before). For any infinite spatial tree $(\tbf,\Snake)$, write $Y_i=\Snake(V_i^\infty(\tbf))$. We write $P_{k,x}^{\mathrm{v\infty}}$ for the law when $\tbf$ has the law of $\Tbf^\mathrm{v}$ and conditionally on $\tbf$, $\Snake$ is a $\tbf$-indexed random walk sending $V^\infty_k$ to $x$.

Let us continue our proof. Note that
\begin{align*}
   &P_{k,x}^\mathrm{v}[y\in\Snake(\{V_0,\dots,V_{k-1}\}\setminus B(k)), V_k\notin B(k),\Bbs(k-l,k)\leq l/2]\\
   =&\sum_{m=0 }^{k-1}P_{k,x}^\mathrm{v}[y\in\Snake(\{V_0,\dots,V_{k-1}\}\setminus B(k)), V_k\notin B(k), |B(k)|=m,\Bbs(k-l,k)\leq l/2]\\
   =&\sum_{m=0 }^{k-1}P_{k-m,x}^{\mathrm{v}\infty}[y\in\Snake(\{V_0,\dots,V_{k-1}\}\setminus B(k)),\\
   &\quad\quad\quad\quad\quad\quad\quad\quad\quad\quad V_k\notin B(k), |B(k)|=m,\Bbs(k-l,k)\leq l/2]\\
   \leq&\sum_{m=0 }^{k-1}P_{k-m,x}^{\mathrm{v}\infty}[y\in\{Y_0,\dots,Y_{k-m-l/2}\},V_k\notin B(k), |B(k)|=m,\Bbs(k-l,k)\leq l/2]\\
   \leq&\sum_{m\leq \lfloor k^{0.6}\rfloor}P_{k-m,x}^{\mathrm{v}\infty}[y\in\{Y_0,\dots,Y_{k-m-l/2}\},|B(k)|=m]+o(1).
\end{align*}
Note that for the third line, we use the fact that conditionally on $|B(k)|=m$ and $V_k\notin B(k)$, $V_k=V^\infty_{k-m}$; for the second last line, that conditionally on $\Bbs(k-l,k)\leq l/2$,
$$
\Snake(\{V_0,\dots,V_{k-1}\}\setminus B(k))\subset\{Y_0,\dots,Y_{k-m-l/2}\};
$$
and for the last line, that $P_{k-m,x}^{\mathrm{v}\infty}[|B(k)|\geq k^{0.6}]=P_{0,x}^{\mathrm{v}\infty}[|B(k)|\geq k^{0.6}]=o(1)$.

Hence, \eqref{rd} can be reduced again to
\begin{equation}\label{rd2}
\lim_{n\rightarrow \infty}\frac{1}{n}\sum_{\epsilon n<k<(1-\epsilon) n}\sum_{m\leq \lfloor k^{0.6}\rfloor}P_{k-m,x}^{\mathrm{v}\infty}[y\in\{Y_0,\dots,Y_{k-m-l/2}\},|B(k)|=m]=0.
\end{equation}
We have (when $n$ is large enough)
\begin{align*}
  &\sum_{\epsilon n<k<(1-\epsilon) n}\sum_{m\leq \lfloor k^{0.6}\rfloor}P_{k-m,x}^{\mathrm{v}\infty}[y\in\{Y_0,\dots,Y_{k-m-l/2}\},|B(k)|=m]\\
  \leq&\sum_{\epsilon n/2<j<(1-\epsilon/2) n}\sum_{m\leq 2\lfloor j^{0.6}\rfloor}P_{j,x}^{\mathrm{v}\infty}[y\in\{Y_0,\dots,Y_{j-l/2}\},|B(k)|=m]\\
  \leq&\sum_{\epsilon n/2<j<(1-\epsilon/2) n}P_{j,x}^{\mathrm{v}\infty}[y\in\{Y_0,\dots,Y_{j-l/2}\}]\\
  =&\sum_{\epsilon n/2<j<(1-\epsilon/2) n}P_{0,x}^{\mathrm{v}\infty}[y\in\{Y_{-j},Y_{-j+1}\dots,Y_{-l/2}\}],
\end{align*}
where for the last step we use Proposition~\ref{invariance}.

Note that $P_{0,x}^{\mathrm{v}\infty}[y\in\{Y_{-j},Y_{-j+1}\dots,Y_{-l/2}\}]\leq P_{0,x}^{\mathrm{v}\infty}[y\in\cup_{i\leq -l/2}\{Y_i\}]$. Since on $\Z^d(d\geq5)$, every finite subset is branching transient (see Section 2 in \cite{Z162}) we get
$\lim_{l\rightarrow \infty}P_{0,x}^{\mathrm{v}\infty}[y\in\cup_{i\leq -l/2}\{Y_{i}\}]=0$ for any $x,y\in\Z^d$. Therefore \eqref{rd2} follows and we finish the proof.
\end{proof}

\begin{proof}[Proof of Lemma~\ref{ll-l3}]
First note that the subtree in $\Tbf^n$ generated by the vertices corresponding to $X_{k-l},\dots,X_{k}$ is contained in $H_l(\Tbf^n,V_k)$ (recall the definition of $H_l$ in Section~4.2), at least when the height of $V_k$ is bigger than $l$. By Theorem~\ref{thm-ll}, we have $H_l(\Tbf^n,V_k)\stackrel{\mathrm{d}}{\rightarrow}H_l(\Tl,u_0)$ (unifromly for $k\in[[\epsilon n,(1-\epsilon)n]]$). Note that, as an unrooted tree, $H_l(\Tl,u_0)$ has the same distribution as $\Tbf^\mathrm{v}$ up to generation $l$. Hence, we have
$$
\lim_{n\rightarrow\infty}\max_{k\in[[\epsilon n,(1-\epsilon)n]]}|P^n_{k,x}[X_{k-l},\dots,X_{k-1}\notin K]-P_{0,x}^{\mathrm{v}\infty}[Y_{-l},\dots,Y_{-1}\notin K]|=0.
$$
Note that the second probability converges to $\Es_K(x)$ when $l\rightarrow \infty$ (recall that $\Es_K(x)=P_x^\mathrm{v}(W_K^0)$). We finish the proof.
\end{proof}
\end{proof}

\subsection{Proof of Theorem \ref{ll-thm-main}}

We have constructed the ingredients we need and are ready to show our main result. We follow the arguments in \cite{ARZ15}.

Let $T$ be the corresponding family tree in $\Snake^{n,N}$. Apply Theorem~\ref{cutting} to $T$, with $(d_0,e,\alpha,\beta,\delta,\epsilon)=(d,2,2.01,d-0.01,0.01,0.01)$ (note that $\epsilon_0=3\beta/2-\alpha-d-7\delta/2=0.5d-2.06\geq0.44$). In fact, the values of $\alpha,\beta,\epsilon,\delta$ are not important. What we essentially need is $\alpha>2$. With high probability ($1-C/N^{\epsilon}$), we can find subtrees $T_1,\dots,T_{m}$ of $T$ as in the theorem. We denote by $A$ this event. We write $P[\bullet|(m;k_1,\dots,k_{m};\tbf)]$ (respectively $p(m;k_1,\dots,k_{m};\tbf)$) for the conditional probability conditioned (respectively the probability) that $A$ is true, the number of $T_i$ is $m$, the size of $T_i$ is $k_i$ ($i=1,\dots,m$) and the subtree $\hat{T}$ with $m$ (ordered) marked vertices indicating the locations of $v_i$ is $\tbf$. Under $P[\cdot|(m;k_1,\dots,k_{m};\tbf)]$, the trees $T_1,\dots,T_{m}$ are independent and distributed as the GW-trees with the given sizes. We have
\begin{align*}
P&[\mathrm{Range}(\Snake^{n,N})\cap \varphi(K)=\emptyset]\\
&=\sum p(m;k_1,\dots,k_{m};\mathbf{t})P[\mathcal{S}(T)\cap \varphi(K)=\emptyset|(m;k_1,\dots,k_{m};\mathbf{t})]+o(1),
\end{align*}
where the sum runs over all possible values of $\Upsilon=(m;k_1,\dots,k_{m};\mathbf{t})$ such that $p(\Upsilon)>0$ (depending on $N$). For notational ease, we simply write $\Snake$ for $\Snake^{n,N}$. It suffices to show
\begin{equation}\label{C3ob}
\lim_{N\rightarrow \infty}\max_{\Upsilon}|P[\mathcal{S}(T)\cap \varphi(K)=\emptyset|\Upsilon]-\exp\left(-u\BCap(K)\right)|=0.
\end{equation}

The equality above can be reduced to:
\begin{multline}\label{C3o1}
\lim_{N\rightarrow \infty}\max_{\Upsilon}|P[\mathcal{S}(T)\cap \varphi(K)=\emptyset|\Upsilon]-\\
P\left[\left(\cup_{i=1}^{m}\mathcal{S}(T_i)\right)\cap \varphi(K)=\emptyset|\Upsilon\right]|=0;
\end{multline}
\begin{multline}\label{C3o2}
\lim_{N\rightarrow \infty}\max_{\Upsilon}|P\left[\left(\cup_{i=1}^{m}\mathcal{S}(T_i)\right)\cap \varphi(K)=\emptyset|\Upsilon\right]-\\
\prod_{i=1}^{m}P\left[\left(\mathcal{S}(T_i)\right)\cap \varphi(K)=\emptyset|\Upsilon\right]|=0;
\end{multline}
\begin{equation}\label{C3o3}
\lim_{N\rightarrow \infty}\max_{\Upsilon}\left|\prod_{i=1}^{m}P\left[\left(\mathcal{S}(T_i)\right)\cap \varphi(K)=\emptyset|\Upsilon\right]-\exp\left(-u\BCap(K)\right)\right|=0.
\end{equation}

The proof of \eqref{C3o1} is easy.
\begin{align*}
&|P[\mathcal{S}(T)\cap \varphi(K)=\emptyset|\Upsilon]-P\left[\left(\cup_{i=1}^{m}\mathcal{S}(T_i)\right)\cap \varphi(K)=\emptyset|\Upsilon\right]|\\
\leq &|P\left[\mathcal{S}(T\setminus \left(\cup_{i=1}^{m}\mathcal{S}(T_i)\right))\cap \varphi(K)\neq\emptyset|\Upsilon\right]|\leq \frac{n}{N^\delta}\frac{|K|}{N^d}\rightarrow 0.
\end{align*}
The last inequality is due to Condition (2) in Theorem~\ref{cutting}, and the fact that $\Snake(v)$ is uniformly distributed in $\Torus_N$ for all $v\in \hat{T}$.

For \eqref{C3o3}, by Condition (1) and (4) in Theorem~\ref{cutting}, we know that $|T_i|\in[N^{a_1},N^{a_2}]$ (for some $a_1,a_2\in(0,d)$) and that conditionally on the size, $T_i$ is a GW-tree conditioned on that size. Hence we can apply Theorem~\ref{visit-small}. Then together with Condition (2), one can get \eqref{C3o3}.

We now turn to \eqref{C3o2}. We need the following lemma.
\begin{lemma}\label{C3o}
There exist positive $c$ and $C$, such that, for any $N\in\N^+$ and  $\Upsilon=(m;k_1,\dots,k_{m};\mathbf{t})$ with $p(\Upsilon)>0$,$k\in[[1,m-1]]$, then
\begin{multline}\label{C3ob2}
|P\left[\left(\cup_{i=k}^{m}\mathcal{S}(T_i)\right)\cap \varphi(K)=\emptyset|\Upsilon\right]-
P\left[\left(\mathcal{S}(T_k)\right)\cap \varphi(K)=\emptyset|\Upsilon\right]\times\\
P\left[\left(\cup_{i=k+1}^{m}\mathcal{S}(T_i)\right)\cap \varphi(K)=\emptyset|\Upsilon\right]|\leq C\exp(-cN^{\alpha-2}).
\end{multline}
\end{lemma}
With this lemma one can use induction to show
\begin{multline}
|P\left[\left(\cup_{i=1}^{m}\mathcal{S}(T_i)\right)\cap \varphi(K)=\emptyset|\Upsilon\right]-\prod_{i=1}^{m}P[\mathcal{S}(T_i)\cap \varphi(K)=\emptyset|\Upsilon]|\\
\leq (m-1)C\exp(-cN^{\alpha-2}).
\end{multline}
Since $m$ is bounded by a polynomial of $N$, the right hand side tends to 0, which implies \eqref{C3o2}.

\begin{proof}[Proof of Lemma \ref{C3o}]
Let $o_1$ and $o_2$ be the endpoints of $\iota_k$ in Theorem~\ref{cutting} (say $o_1\in T_k$). For any $x,y\in \T_N$, define
\begin{align}
f(x)&=P\left[\left(\mathcal{S}(T_k)\right)\cap \varphi(K)=\emptyset|\mathcal{S}(o_1)=x,\Upsilon\right],\\
h(y)&=P\left[\left(\cup_{i=k+1}^{m}\mathcal{S}(T_i)\right)\cap \varphi(K)=\emptyset|\mathcal{S}(o_2)=y,\Upsilon\right].
\end{align}
By Condition (3) in Theorem~\ref{cutting}, $\iota_k$ separates $T_k$ and $\cup_{i=k+1}^{m}T_i$, so we have
\begin{equation}
P\left[\left(\cup_{i=k}^{m}\mathcal{S}(T_i)\right)\cap \varphi(K)=\emptyset|
\mathcal{S}(o_1)=x,\mathcal{S}(o_2)=y,\Upsilon\right]=f(x)\times h(y).
\end{equation}
Therefore,
\begin{align*}
P\left[\left(\cup_{i=k}^{m}\mathcal{S}(T_i)\right)\cap \varphi(K)=\emptyset|\Upsilon\right]&=\sum_{x,y\in \T_N}f(x)h(y)P[\mathcal{S}(o_1)=x,\mathcal{S}(o_2)=y|\Upsilon]\\
&=N^{-d}\cdot\sum_{x,y\in \T_N}f(x)h(y)P^{\mathrm{RW}}_x[Z(\lfloor N^{\alpha}\rfloor)=y],
\end{align*}
where $P^{\mathrm{RW}}_x$ is the law of $Z=(Z(n))_{n\in\N}$ which is a random walk starting at $x$ with distribution $\theta$ in $\Torus_N$.
Note that
\begin{eqnarray}
P\left[\left(\mathcal{S}(T_k)\right)\cap \varphi(K)=\emptyset|\Upsilon\right]=
N^{-d}\sum_{x\in\T_N}f(x);\\
P\left[\left(\cup_{i=k+1}^{m}\mathcal{S}(T_i)\right)\cap \varphi(K)=\emptyset|\Upsilon\right]=
N^{-d}\sum_{y\in\T_N}h(y).
\end{eqnarray}
Hence the left hand side of \eqref{C3ob2} is
\begin{multline*}
|N^{-d}\sum_{x,y\in \T_N}f(x)h(y)(P^{\mathrm{RW}}_x[Z(\lfloor N^{\alpha}\rfloor)=y]-N^{-d})|\\
\leq \max_{x\in \T_N}\sum_{y\in \T_N}|P^{\mathrm{RW}}_x[Z(\lfloor N^{\alpha}\rfloor)=y]-N^{-d}|.
\end{multline*}
Now \eqref{C3ob2} follows from the following standard result in the theory of mixing time (e.g.\ see \cite{LP17M}).
\begin{lemma}\label{mixing}
There exist positive numbers $c$ and $C$ such that for any $N\in\N^+$ and $a>2$,
\begin{equation}
\max_{x,y\in \T_N}|P^{\mathrm{RW}}_x[Z(\lfloor N^{a}\rfloor)=y]-N^{-d}|\leq C\exp(-cN^{a-2}).
\end{equation}
\end{lemma}
\end{proof}

\section{Cover times of tori by tree-indexed random walks}
The main goal of this section is to construct the result of cover times of $d$-dimensional tori by tree-indexed random walks conditioned on sizes. Recall that we adopt the same assumptions as in the last section on the offspring distribution $\mu$ and the jump distribution $\theta$ (and $d\geq 5$). As before, write $\Snake^{n,N}$ for the tree-indexed random walk in torus $\Torus_N$, with size $n$ and a uniform starting point.

\begin{theorem}\label{thm-cover}
Let $n=n(N)$ be an integer-valued function of $N$. For any $\epsilon\in(0,1)$, we have if $n(N)>(1+\epsilon)N^d\log N^d/\BCap(\{0\})$, then
$$
\lim_{N\rightarrow\infty}P[\mathrm{Range}(\Snake^{n,N})=\Torus_N]=1;
$$
if $n(N)<(1-\epsilon)N^d\log N^d/\BCap(\{0\})$, then
$$
\lim_{N\rightarrow\infty}P[\mathrm{Range}(\Snake^{n,N})=\Torus_N]=0.
$$
\end{theorem}

We need two lemmas in our proof of the theorem above. Theorem~\ref{visit-small} gives asymptotics for the probability that a set is visited by a tree-indexed random walk with a uniform starting point. The following proposition gives asymptotics for a tree-indexed random walk with a fixed starting point, which is not too close to the set.

\begin{prop}\label{visit-small-fixed}
For any $\beta_1,\beta_2,\gamma$ fixed with $(2d+4)/3<\beta_1<\beta_2<d$, $(3d-3\beta_1+4)/d<\gamma<1$, $n=n(N)$, any integer-valued function of $N$ satisfying $n(N)\in[N^{\beta_1},N^{\beta_2}]$, and $K\ssubset \Z^d$, we have
\begin{equation}\label{eq-visit-small-fixed}
    \lim_{N\rightarrow\infty}\sup_{x\in\Torus_N:\dist(x,\varphi(K))>N^\gamma}|\frac{N^d}{n}P[\mathrm{Range}(\Snake)\cap \varphi(K)\neq\emptyset]-\BCap(K)|=0,
\end{equation}
where $\Snake=\Snake_x^{n,N}$ is a tree-indexed random walk in $\Torus_N$ with size $n$ and starting point $x$ and $\dist$ is the graph distance of the torus.
\end{prop}
\begin{proof}
Fix any $x$ with $\dist(x,\varphi(K))>N^\gamma$. By our assumptions on $\beta_1,\gamma$, we can find $\alpha>0$ satisfying
\begin{equation}\label{eq-alpha}
\beta_1+d\gamma>4+2\alpha+d, \;\alpha+\beta_1>d.
\end{equation}
Write $T$ for the corresponding tree in $\Snake$. Applying Proposition~\ref{cut-2} with $k=\lfloor N^{2+c}\rfloor$, $L=\lfloor k^2N^{2\alpha}\rfloor$ to $T$, we obtain that, with probability at least $1-CN^{-(\alpha-c)}$ ($C$ may depend on $\beta_1,\beta_2,\gamma,c$), there exists a subtree $T_0$ which is at distance $k$ away from the root of $T$ without losing more than $L$ vertices. Assume that this event has occured and condition on the value of $T\setminus T_0$. Write $o_0$ for the root of $T_0$. Similarly to the situation in the proof of Theorem~\ref{visit-small}, one can see that the law of $\Snake(o_0)$ is very close to the uniform measure. Therefore, by Theorem~\ref{visit-small}, we have
$$
P[\Snake(T_0)\cap \varphi(K)\neq\emptyset]\sim\frac{|T_0|}{N^d}\BCap(K)\sim\frac{n}{N^d}\BCap(K).
$$
We point out that
\begin{equation}\label{range-error}
P[\Snake(T \setminus T_0)\cap\varphi(K)\neq\emptyset]\preceq \frac{L|K|}{N^{d\gamma}}.
\end{equation}
If so, then we finish the proof since by \eqref{eq-alpha}, we have (when choosing $c$ small enough)
$$
N^{-(\alpha-c)}\preceq\frac{n}{N^{d+o(1)}},\quad \frac{L|K|}{N^{d\gamma}}\preceq \frac{n}{N^{d+o(1)}}.
$$

We still need to show \eqref{range-error}. It can be obtained from the following. For any $y\in\Torus_N$ with $\dist(x,y)>N^\gamma$ and $v\in T\setminus T_0$, we have
$$
P[\Snake(v)=y]\leq \sup_{n\in\N}P[Z_x(n)=y]\preceq N^{-d\gamma},
$$
where $Z_x=(Z_x(n))_{n\in N}$ is a random walk starting at $x$ in $\Torus_N$, and the last step follows from the Local Central Limit Theorem.
\end{proof}

\begin{lemma}\label{range-xy}
For any $\beta_1<\beta_2\in(0,d)$, $\epsilon>0$ and $n=n(N)$, any integer-valued function of $N$ such that $n(N)\in[N^{\beta_1},N^{\beta_2}]$, there exists $C=C(\epsilon)$, such that for any $x\neq y\in\Torus_N$,
$$
P[\{x,y\}\subset\mathrm{Range}(\Snake^{n,N})]\leq C\frac{n}{N^d}(\frac{1}{(\dist(x,y))^{d-4}}+\frac{(\log N)^{3+\epsilon}}{N^{d-\beta_2}}).
$$
\end{lemma}
\begin{proof}
Note that the left hand side is no greater than
\begin{align*}
 \sum_{0\leq i<j\leq n-1}&(P^{n,N}[X_i=x,X_j=y]+P^{n,N}[X_i=y,X_j=x])\\
&=2\sum_{0\leq i<j\leq n-1}P^{n,N}[X_i=x,X_j=y].
\end{align*}
Hence, we have
\begin{align*}
    \frac{N^d}{n}&P[\{x,y\}\subset\mathrm{Range}(\Snake^{n,N})]\preceq \frac{N^d}{n}\sum_{0\leq i<j\leq n-1}P^{n,N}[X_i=x,X_j=y]\\
    =&\frac{1}{n}\sum_{0\leq i<j\leq n-1}P^{n,N}_{i,x}[X_j=y]=\frac{1}{n}\sum_{0\leq i<j\leq n-1}\sum_{k}P^{n,N}_{i,x}[X_j=y,\mathbf{d}(V_i,V_j)=k]\\
    =&\frac{1}{n}\sum_{k}P[Z_x(k)=y]\sum_{0\leq i<j\leq n-1}P^{n,N}_{i,x}[\mathbf{d}(V_i,V_j)=k],
\end{align*}
where $P^{n,N},P^{n,N}_{i,x}$ are similar to the ones in Section~4, $\mathbf{d}(V_i,V_j)$ is the graph distance of the $i$-th vertex $V_i$ and the $j$-vertex $V_j$ in the corresponding tree $\Tbf^n$, and $Z_x=(Z_x(i))_{i\in \N}$ is a random walk in $\Torus_N$ with starting point $x$. From Lemma~\ref{J2}, we have
$$
\frac{1}{n}\sum_{0\leq i<j\leq n-1}P^{n,N}_{i,x}[\mathbf{d}(V_i,V_j)=k]\preceq k.
$$
and therefore
$$
\frac{N^d}{n}P[\{x,y\}\subset\mathrm{Range}(\Snake)]\leq \sum_k kP[Z_x(k)=y].
$$
Note that when $x\neq y\in\Z^d$,
\begin{equation}\label{kP}
   \sum_kP[Z_x(k)=y]\cdot k\preceq|x-y|^{4-d}.
\end{equation}
This is false for $x\neq y\in\Torus_N$ since when $x,y\in\Z^d$ with $\varphi(x)\neq\varphi(y)$,
$$
P[Z_{\varphi(x)}(k)=\varphi(y)]=\sum_{z\in\Z^d}P[Z_x(k)=y+Nz].
$$
However, similarly to the argument in the proof of Lemma~\ref{ll-l1}, one can use Lemma~\ref{large-range} to rule out those $z$ with norm bigger than $b=N^{\beta_2/4-1}(\log N)^{0.75+o(1)}$. The sum over the remaining $z$ has the desired bound: \\
(without loss of generality, assume that $y$ achieves $\min_{y':\varphi(y')=\varphi(y)}\{|x-y'|\}$)
\begin{align*}
&\sum_{z\in\Z^d:|z|\leq b}\sum_k kP[Z_x(k)=y+Nz]\asymp\sum_{z\in\Z^d:|z|\leq b}|x-y-Nz|^{4-d}\\
\asymp&|x-y|^{4-d}+\sum_{i=1}^b\frac{i^{d-1}}{(iN)^{d-4}}\asymp \frac{1}{(\dist(x,y))^{d-4}}+\frac{C(\epsilon)(\log N)^{3+\epsilon}}{N^{d-\beta}}.
\end{align*}
The proof is complete.
\end{proof}
\begin{proof}[Proof of Theorem~\ref{thm-cover}]
Write $a_0=\BCap(\{0\})$ and $\Snake=\Snake^{n,N}$. We start with the upper bound. Assume that $n(N)>(1+\epsilon)N^d \log N^d/a_0$. Let $T$ be the corresponding family tree in $\Snake$.  Apply Theorem~\ref{cutting} with $(d_0,e,\alpha,\beta,\delta,\epsilon)=(d,2,2.01,d-0.01,0.01,0.01)$. Then, by discarding an event with small probability, we can find subtrees $T_1,\dots,T_m$ as in the theorem. Write $l_1,\dots,l_m$ for their sizes. By Theorem~\ref{visit-small}, we have (when $N$ is large enough)
$$
P[x\in \Snake(T_i)]\geq (1-0.1\epsilon)\frac{l_i}{N^d}a_0,
$$
for any $x\in\Torus_N$.
Similarly to the argument in the last section, by Lemma~\ref{C3o}, one can get
$$
P[x\notin \cup_{i=1}^{m}\Snake(T_i)]\leq \prod_{i=1}^m P[x\notin \Snake(T_i)]+C\exp(-cN^{\alpha-2}).
$$
Note that (when $N$ is large)
\begin{multline*}
\prod_{i=1}^m P[x\notin \Snake(T_i)]\leq\prod_{i=1}^m(1-(1-0.1\epsilon)\frac{l_i}{N^d}a_0)\leq\exp(-(1-0.2\epsilon)\frac{\sum_{i=1}^ml_i}{N^d}a_0)\\
\leq \exp(-(1-0.3\epsilon)(1+\epsilon)\log (N^d))\leq \exp(-(1+0.6\epsilon)\log (N^d)).
\end{multline*}
Hence we have
$$
P[x\notin \cup_{i=1}^{m}\Snake(T_i)]\leq \frac{1}{N^{d(1+0.5\epsilon)}}.
$$
Since there are only $N^d$ vertices in $\Torus_N$, one can get $P[\mathrm{Range}(\Snake)\neq\Torus_N]\leq N^{-0.5\epsilon d}$. This finishes the proof of the upper bound.

We now turn to the lower bound. Fix some $\eta<\epsilon \wedge 0.01$. Apply Theorem~\ref{cutting} with
$$
(d_0,e,\alpha,\beta,\delta,\epsilon)=(d,2,2+\cc,d-\cc,\cc,\cc),
$$
where $\cc$ is a small number (in fact $9\cc<0.5$ is enough).
By discarding an event with small probability, we can assume that there exist subtrees $T_1,\dots,T_m$ as in the theorem. Write $l_1,\dots,l_m$ for their sizes and $o_1,\dots,o_m$ for their roots. We condition on the values of the subtree $\hat{T}$ (recall this notation in the theorem) and $\Snake|_{\hat{T}}$ (and the values of $l_1,\dots,l_m$). Note that conditionally on these vaules, $(T_1,\Snake|_{T_1}),\dots,(T_m,\Snake|_{T_m})$ are independent tree-indexed random walks with given sizes and starting points. By Proposition~\ref{visit-small-fixed} with $\beta_1=\beta-2\cc-o(1)$ and $\beta_2=\beta$ , we see that as long as $\gamma>(9\cc+4)/d$, \eqref{eq-visit-small-fixed} holds for $\Snake(T_i)$. Let $\gamma=4.5/d$. Then we have \eqref{eq-visit-small-fixed}.

We claim that we can pick up $\lfloor N^{(1-\eta)d}\rfloor$ vertices (write $F$ for the set of these vertices) in $\Torus_N$, such that the distance between any two vertices in $F$ is at least $N^{\eta}/8$, $F\cap(\Snake(\hat{T}))=\emptyset$ and $\dist(F,\Snake(\cup\{o_i\}))>N^{\gamma}$. Here is a way to do so. First we can find at least $4N^{(1-\eta)d}$ boxes with radius $N^{\eta}/8$ in $\Torus_N$ such that the distance between any two in them are at least $N^{\eta}/8$. Write $B_0$ for the set of these boxes. Note that for each $y\in A_1\doteq \Snake(\cup\{o_i\})$, there are at most $(32N^\gamma/N^{\eta})^d\asymp N^{4.5-d\eta}$ boxes in $B_0$ whose distance from $y$ is at most $N^\gamma$. Since there are only $m$ ($m\leq N^{4\cc}$) vertices in $A_1$, there are at least $3N^{(1-\eta)d}$ boxes in $B_0$ which are at least $N^\gamma$ away from $A_1$. On the other hand, note that there are at most $n/N^\cc\ll N^d$ vertices in $\Snake(\hat{T})$ and that a strictly positive proportion of vertices in $\Torus_N$ are covered by the remaining $3N^{(1-\eta)d}$ boxes. Hence, most of the $3N^{(1-\eta)d}$ boxes are not fully covered by $\Snake(\hat{T})$ and in each such box we can pick up a point not in $\Snake(\hat{T})$. In this way, we obtain $F$ with the desired properties.

We argue that with high probability, $F$ is not covered by $R\doteq\Snake(\cup_i^mT_i)$, and therefore $F$ is not covered by $\Snake(T)$. Then we finish the proof.

Without loss of generality, assume that $\sum_{i=1}^ml_i\sim (1-\epsilon)N^d\log N^d/a_0$. Write $U=\sum_{x\in F}1_{x\notin R}$. We need to show that it is positive with high probability. We manage to do so by using the Chebyshev inequality and hence need to estimate the expectation and the variance of $U$. First, note that (since $(\Snake(T_i))_{i=1,\dots,m}$ are independent)
$$
E[1_{x\notin R}]=P[x\notin \Snake(\cup_i^mT_i)]=\prod_i^mP[x\notin \Snake(T_i)]=\prod^m_i(1-P[x\in \Snake(T_i)]).
$$
Using \eqref{eq-visit-small-fixed}, we have
$$
\prod^m_i(1-P[x\in \Snake(T_i)])
=\exp(-\frac{(1+o(1))a_0\sum_i^m l_i}{N^d})=\frac{1}{N^{(1-\epsilon)d+o(1)}}.
$$
Hence, we obtain $E[U]=|F|/N^{(1-\epsilon)d+o(1)}=N^{d(\epsilon-\eta)+o(1)}$.

We now turn to the variance. Note that $\mathrm{Var}[U]=\sum_{x,y\in F}q_{x,y}\leq E[U]+\sum_{x\neq y\in F}q_{x,y}$ where $q_{x,y}=P[x,y\notin R]-P[x\notin R]P[y\notin R]$.
Note that, as before, by independence, we have
\begin{eqnarray*}
P[x,y\notin R]=\prod_i^mP[x,y\notin\Snake(T_i)],\\
P[x\notin R]P[y\notin R]=\prod_i^m P[x\notin\Snake(T_i)] P[y\notin\Snake(T_i)].
\end{eqnarray*}
On the other hand,
$$
P[x,y\notin\Snake(T_i)]=1-P[x\in\Snake(T_i)]-P[y\in\Snake(T_i)]+P[x,y\in\Snake(T_i)].
$$
We point out an inequality and show it later: for any $x\neq y\in F$
\begin{equation}\label{prob-xy}
\frac{N^d}{l_i}P[x,y\in\Snake(T_i)]\leq \frac{C}{N^c},
\end{equation}
where $C,c$ are constants which may depend on $\cc,\eta,\epsilon$.
Therefore, we have
\begin{align*}
\frac{P[x,y\notin R]}{P[x\notin R]P[y\notin R]}\leq&\prod_i^m \frac{1-P[x\in\Snake(T_i)]-P[y\in\Snake(T_i)]+P[x,y\in\Snake(T_i)]}{(1-P[x\in\Snake(T_i)])(1-P[y\in\Snake(T_i)])}\\
\leq& \prod_i^m(1+C\frac{l_i}{N^{d+c}})\leq 1+\frac{C}{N^c}.
\end{align*}
From this, we get $q_{x,y}\preceq P[x\notin R]P[y\notin R]/N^c$ and hence
$$
\sum_{x\neq y\in F}q_{x,y}\preceq (\sum_xP[x\notin R])^2/N^c=(E[U])^2N^{-c}.
$$
By the Chebyshev inequality, we have
$$
P[|U-E[U]|>\frac{E[U]}{2}]\leq\frac{4\mathrm{Var}[U]}{(E[U])^{2}}\preceq\frac{1}{N^{c/2}}.
$$
Hence, at least $E[U]/2$ vertices in $F$ are not covered by $\Snake(\cup_i^mT_i)$ with high probability.

We still need to show \eqref{prob-xy}. The argument is almost the same as the one in the proof of Proposition~\ref{visit-small-fixed} except that we need to use Lemma~\ref{range-xy} instead of Theorem~\ref{visit-small}. As in the proof of Proposition~\ref{visit-small-fixed}, (recall that we have verified that our $\gamma$ and the size $l_i$ satisfy the requirements in that proposition, see the paragraph after the proof of the upper bound) by discarding an event with probability less than $Cl_i/N^{d+c}$, we can find a subtree $T_0$ of $T_i$ as in the proof of proposition. Similarly, we have
$$
P[\Snake(T_i\setminus T_0)\cap\{x,y\}\neq\emptyset]\preceq \frac{l_i}{N^{d+c}}.
$$
For $\Snake(T_0)$, its starting point is very close to the uniform measure hence we can apply Lemma~\ref{range-xy} and obtain
$$
P[x,y\in \Snake(T_0)]\preceq \frac{l_i}{N^{d+c}}.
$$
Now \eqref{prob-xy} follows and the proof is complete.
\end{proof}

\section*{Acknowledgements}
This work was partially done when the author was a PhD student in the University of British Columbia. The author would like to thank Omer Angel and Bal\'{a}zs R\'{a}th for supervising his research and useful discussions.

\end{document}